\documentclass{article}
\usepackage{amsfonts}
\usepackage{amsthm}
\usepackage{amsmath}
\usepackage{amssymb}
\usepackage{graphicx}%
\usepackage{epsfig}
\usepackage{color}
\usepackage{dsfont}
\providecommand{\U}[1]{\protect\rule{.1in}{.1in}}

\def\rit{\hbox{\it I\hskip -2pt R}}
\def\eit{\hbox{\it I\hskip -2pt E}}

\def\subneq{\mathop{\raise 0.7ex \hbox{$\subset$}}\!\!\!\!\!\!{\raise -0.6ex\hbox{$\neq$}}\,}

\def\div{\mathop{\rm div}\nolimits}

\def\u{\underline}

\def\RR{\mathbb{R}}

\def\1{{1\hskip-0.25em{\rm l}}}

\def\sobre#1#2{\lower 1ex \hbox{ $#1 \atop #2 $ } }
\def\bajo#1#2{\raise 1ex \hbox{ $#1 \atop #2 $ } }

\definecolor{darkred}{rgb}{0.9,0.1,0.1}

\def\div{{\rm div}}
\def\ep{\varepsilon}

\def\p{\partial}

\def\O{\Omega}

\def\u2{{u^\ep \over \ep^2 }}
\def\u3{{\displaystyle {\bar u}^\ep \over \ep^2 }}

\def\div{\mathop{\rm div}\nolimits}

\def\u{\underline}

\def\sobre#1#2{\lower 1ex \hbox{ $#1 \atop #2 $ } }
\def\bajo#1#2{\raise 1ex \hbox{ $#1 \atop #2 $ } }

\def\div{{\rm div}}

\def\ep{\varepsilon}

\def\p{\partial}

\def\u2{{u^\ep \over \ep^2 }}
\def\u3{{\displaystyle {\bar u}^\ep \over \ep^2 }}
\def\p{\partial}
\def\pej{\mbox{\rm Pe}_{j}}
\def\pei{\mbox{\rm Pe}_{i}}

\def\dsp{\displaystyle}

\newtheorem{theorem}{Theorem}

\newtheorem{corollary}[theorem]{Corollary}

\newtheorem{definition}[theorem]{Definition}

\newtheorem{lemma}[theorem]{Lemma}

\newtheorem{proposition}[theorem]{Proposition}
\newtheorem{remark}[theorem]{Remark}

\begin{document}
\title{Homogenization of the linearized ionic transport
equations in random porous media\thanks{The research of A.M. was supported in part by  the project 
 UPGEO $\langle$ ANR-19-CU05-032 $\rangle$ of the French National Research Agency (ANR), by the LABEX MILYON (ANR-10-LABX-0070) of
Universit\'e de Lyon, within the program "Investissements d'Avenir" (ANR-11-IDEX-0007) operated by the French National Research Agency (ANR) and  by the project "Ecoulements \'electrocin\'etiques \`a travers un milieu poreux al\'eatoire", GDR MOMAS CNRS 2015.}}
\author{ Andro Mikeli\'c \\ E-mail: {\tt mikelic@univ-lyon1.fr} \\
Univ Lyon, Universit\'e Claude Bernard Lyon 1,
 CNRS UMR 5208, \\
 Institut Camille Jordan, 43 blvd. du 11 novembre 1918, \\
 F-69622 Villeurbanne cedex, France
 \and  Andrey Piatnitski\\ E-mail: {\tt apiatnitski@gmail.com} \\
              The Arctic University of Norway, campus Narvik, Postbox 385, Narvik 8505, Norway\\
             and \\ Institute for Information Transmission Problems of RAS, Bolshoi Karetny per., 19, \\ 127051, Moscow, Russia
}
\date{\it \normalsize This contribution is dedicated to the memory of Professor Vassily Jikov. Stochastic homogenization and partial differential equations were between his favorite subjects and we give homage to his works.}

\maketitle

\begin{abstract}
In this paper we extend the homogenization results obtained in \cite{AllMiPi:10} (G. Allaire, A. Mikeli\'c, A. Piatnitski, J. Math. Phys. 51 (2010), 123103)  for a system of partial differential equations describing the transport of a N-component electrolyte in a dilute Newtonian solvent through a rigid periodic porous medium, to the case of random disperse porous media. We present a study of the nonlinear Poisson-Boltzmann equation in a random medium, establish convergence of the stochastic homogenization procedure and prove well-posedness of the two-scale homogenized equations. In addition, after separating scales, we prove that the effective tensor satisfies the so-called Onsager properties, that is the tensor is symmetric and positive definite. This result shows that the Onsager theory applies to random porous media. The strong convergence of the fluxes is also established.
\end{abstract}


\section{Introduction}\label{Intro}
The quasi-static electroosmotic phenomena in porous media are present  in many problems of applied interest. As examples we mention electromigration of solutes (see for instance Ottosen et al \cite{Ottosen}), dewatering (see Mahmoud et al \cite{Mahmoud}) and permeability reduction in concretes (see Cardenas \& Struble \cite{Cardenas}).

We consider a porous material saturated by
 an $N$-component dilute electrolyte. The solid surfaces of the porous skeleton are electrically charged, attract ions of the opposite charge and repel the ones of the same charge. 
 Simultaneously an external electrical field $\mathbf{E}$ 
 and a body force $\mathbf{f}$ are applied, 
 generating a hydrodynamical flow, a migration of ions and creation of an electrical double layer (EDL).

The modeling of phenomena at the pore level is well understood (Lyklema \cite{Lyklema}, Karniadakis et al  \cite{KBA:05}).
 The porous media are heterogeneous and have a pore structure consisting of a very large number of pores. In studying flows generated by the electro-osmotic phenomena, the pore size is of the same order as the size of the Debye layer ($\approx 100$ nanometers). Solving the partial differential equations of the model is a difficult task even for simple geometries and out of reach for realistic porous media with nanoscopic pores. The remedy for the complexity of the problem is to homogenize or upscale the equations posed 
 at the pore scale and derive a new upscaled system
 valid at every point. The microscopic geometry would influence the effective coefficients.

The homogenization approach allows finding out how close the solutions to upscaled model are to  those of the original physical equations, given at the pore scale. In the existing literature on the charged porous media, the homogenization technique was mainly applied under the hypothesis of the {\it periodic} porous media and for the so-called {\it ideal model, i.e. with ions at infinite dilution}.
There are works which concern the study and homogenization of the Nernst-Planck-Poisson system (see for instance Ray et al \cite{Rayetal:12}-\cite{Rayetal:12b} and Schmuck et al \cite{schmuck09}-\cite{schmuckBaz}). In the physical chemistry literature, its semi-linearized form, due to O'Brien \& White \cite{OBW:78}, replaces the full model from Lyklema \cite{Lyklema} and Karniadakis et al  \cite{KBA:05}. Looker \& Carnie  studied  O'Brien's model using the two-scale expansion in \cite{LC:06}. They derived  constitutive laws linking the effective fluxes with the pressure gradient and the gradient of the chemical potential. In addition, the Onsager tensor, containing the effective coefficients, was obtained.   The   rigorous convergence result for the homogenization process is due to Allaire et al \cite{AllMiPi:10}, where the positive definiteness of the Onsager tensor and its symmetry were proved too.  A numerical and qualitative analysis of the effective coefficients is in Allaire et al \cite{AllBrizzDufMiPi:12}. The results of Looker \& Carnie \cite{LC:06} and Allaire et al \cite{AllMiPi:10}, \cite{AllBrizzDufMiPi:12} support findings in the earlier work of
 Adler et al \cite{CSTA:96}, \cite{MSA:01},  \cite{AM:03} and  \cite{GCA:06}, on the particular flows and the computation of the effective coefficients forming  Onsager's tensor.

The electro-osmotic phenomena are important also in {\it deformable} porous media, like clays.
For the homogenization studies of the electro-osmosis with swelling, using two-scale expansions,
we refer to the series of articles by Moyne \& Murad
\cite{MM:02}, \cite{MM:03}, \cite{MM:06}, \cite{MM:06a} and \cite{MM:08}.  A further reference is
 the work of Dormieux et al \cite{Dorm:03}. A  rigorous homogenization of the electro-osmotic flows in deformable porous media, with a derivation of the Onsager relation and determination of the swelling pressure, was obtained in Allaire et al \cite{ABDM2015}.

We note that models that are more realistic  involve electrolytes at {\it finite dilution}. Several complicated mathematical models
were developed to take into account the finite size of the ions.  One of them is the {\it Mean Spherical Approximation (MSA)} model (Dufr\^eche et al \cite{Dufreche05}). Its well-posedness and homogenization  was  studied by Allaire et al \cite{AllBrizzDufMiPi:14}.

Finally, the effects of nonlinearities were studied for equilibrium solution (Ern et al \cite{ern}, Allaire et al \cite{beta}).


The realistic soils are {\it random geometric structures} and questions related to the upscaling of the electro-osmotic flows were  studied only through numerical simulations.  We mention the publications of Adler and al in \cite{CSTA:96}, \cite{MSA:01} and \cite{AM:03}, where the electrokinetic flows through random packings of spheres and ellipsoids were considered.

The goal of the present paper is to study {\it upscaling of the  ideal model
describing the transport of a dilute $N$-component electrolyte in
a rigid {\bf random} porous medium}. We first briefly recall the dimensionless equations, which were the starting point for the periodic homogenization in \cite{AllMiPi:10}, \cite{AllBrizzDufMiPi:12} and \cite{AllBrizzDufMiPi:14}. Since the modeling and derivation of the dimensionless form were undertaken in detail in these references (and especially in \cite{AllBrizzDufMiPi:14}) we do not dwell on the subject and simply
start with the dimensionless ideal model equations  from the above references.


The (dimensionless) equations are given on a typical realization of the porous medium $G$. If $\omega$ takes values in the probability space, we  denote by $ G_f^\varepsilon(\omega)$ a realization of the pore space, which is an open set filled with a fluid. The equations then read as follows:
\begin{gather}
\ep^2 \Delta \mathbf{u}^\ep - \nabla p^\ep = \mathbf{f}^* - \sum_{j=1}^N z_j n_j^\ep (x) \nabla\Psi^\ep    \ \mbox{ in } G_f^\varepsilon(\omega), \label{EPPR1} \\
\mathbf{u}^\ep =0 \ \mbox{ on } \, \p G_f^\varepsilon(\omega), \quad
\div \ \mathbf{u}^\ep =0 \quad \mbox{in } \, G_f^\varepsilon(\omega), \label{EPPR6}\\
 -\ep^2 \Delta \Psi^\ep = \beta \sum_{j=1}^N  z_j n_j^\ep (x) \quad \mbox{in } \ G_f^\varepsilon(\omega);  \label{EPPR3b}\\
\ep \nabla \Psi^\ep \cdot \nu = -N_\sigma \sigma  \; \mbox{ on } \, \partial G_f^\varepsilon(\omega) \setminus \p G,\quad \Psi^\ep  = - \Psi^{ext}  \; \mbox{ on } \, \partial G,
   \label{EPPR2} \\
\mbox{div } \big(  n_j^\ep \nabla \ln (n_j^\ep  e^{\Psi^\ep z_j } ) - \pej n_j^\ep \mathbf{u}^\ep  \big) =0 \ \mbox{ in } \; G_f^\varepsilon(\omega), \label{Nernst1} \\
\nabla \ln (n_j^\ep  e^{\Psi^\ep z_j } ) \cdot \nu = 0 \; \mbox{ on } \, \partial G_f^\varepsilon(\omega).\label{EPPR5}
  \end{gather}

System (\ref{EPPR1})-(\ref{EPPR5}) is the dimensionless  model
that we will homogenize in the sequel. {
More precisely, we will undertake study of its O'Brien's semi-linearized form.}  We assume that all
constants appearing in (\ref{EPPR1})-(\ref{EPPR5}) are independent of $\ep$,
namely $N_\sigma$ and $\pej$ are of order 1 with respect to $\ep$.

For the comfort of the reader, we recall meaning of the unknowns (which are dimensionless in our equations) on Table \ref{Data}.

\begin{table}[ht]
\footnotesize
\centerline{\begin{tabular}{|l|l|l|} \hline\hline
\emph{SYMBOL} & \emph{ QUANTITY }
 \\
 \hline  $\pej$ & is the P\'eclet number for $j$th electrolyte component \\
\hline    $n_i$   & $i$th  concentration     \\
\hline    $n_i^c \in (0,1)$   & $i$th  infinite dilution concentration     \\
\hline    $\mathbf{u}$  & fluid velocity   \\
\hline    $p$  & fluid pressure   \\
\hline   $\ell$   & pore size     \\
\hline    $\lambda_D$  & Debye's length     \\
\hline     $\beta =\displaystyle (\ell / \lambda_D)^2$ & ratio of the pore scale to the Debye's length \\
\hline     $N_\sigma$  & ratio of the pore scale to the Gouy length \\
\hline     $z_j \in \mathbb{Z} $   &  $j$-th electrolyte valence    \\
\hline    $\sigma$   & surface charge density     \\
\hline    $\mathbf{f}^*$  &  given applied force  \\
\hline    $\Psi$   &  electrokinetic potential  \\
\hline    $\mathbf{E}=\nabla \Psi^{ext}$ & exterior electrical field generated by the exterior potential $\Psi^{ext}$\\
\hline
\end{tabular}
} \caption{{\it Description of the parameters and the unknowns}\label{Data}}
\end{table}

In Section \ref{Poiss}, we recall in subsection \ref{2.1} the basic results on the stochastic two-scale convergence technique, which will be used to prove convergence of the homogenization process in random geometry. In subsection \ref{2.2} their partial linearization in the spirit of the seminal work of O'Brien \& White \cite{OBW:78} is given. The section is completed with subsection \ref{PerioPB}, where the well-posedness of the non-linear Poisson-Boltzmann equation in a random geometry is studied (see Theorem \ref{Limiteqequilib}). In Section \ref{Passlimit}, we undertake stochastic homogenization of the linearized electrokinetic equations around equilibrium (Theorem \ref{1.15}). The scale separation and derivation of Onsager's relations, linking the ionic current, filtration velocity and ionic fluxes with gradients of the electrical potential, pressure and ionic concentrations, are the subject of Section \ref{sec4}. In Proposition \ref{prop.eff} we prove that the full homogenized Onsager tensor is symmetric positive definite for disperse random structures. The article is concluded with Section \ref{StrongS}, where in Theorem \ref{thmstrong} we prove strong convergence of the velocities and the ionic fluxes.

\section{Description of the problem and of the results}\label{Poiss}

In this section we give a precise formulation of our microscopic problem (the
$\varepsilon$-problem) and of the results.

We start with defining the geometrical structure

\subsection{The stochastic two-scale convergence in mean and a random porous medium}\label{2.1}

Let $(\Omega,\,\Xi,\,\mu)$  denote a probability space, with probability measure $\mu$ and sigma-algebra $\Xi$. 
In what follows $L^2(\Omega)$ is supposed to be separable.
We assume that
an ergodic dynamical system $\cal T$ with n-dimensional time
is given on $\Omega$, i.e. a family of
invertible measurable maps ${\cal T}(x)\,:\,\Omega\to\Omega$, for each $x\in\mathbb{R}^n$,
such that
\begin{enumerate}
\item ${\cal T}(0) \hbox{ is the identity map on } \Omega \hbox{\ and\ } {\cal T}(x_1+x_2)=
{\cal T}(x_1 ){\cal T}(x_2 )$\ for all $x_1,x_2\in\mathbb{R}^n$;
\item $\forall\, x\in\mathbb{R}^n$ and $\forall\, E\in\Xi$,
$$\mu({\cal T}(x)^{-1}(E)) = \mu(E), \quad \mbox{i.e. $\mu$ is an invariant measure for } \, \mathcal{T} \,
(\mbox{endomorphism property}).$$
\item $\forall\, E\in\Xi$\ the set\ $\{ (x,\omega) \in\mathbb{R}^n
\times\Omega\,:\,{\cal T}(x)\omega\in E\}$
is an element of the sigma-algebra ${\cal L}\times\Xi$\ on\
$\mathbb{R}^n\times\Omega$, where ${\cal L}$
is the usual Lebesgue $\sigma$-algebra on $\mathbb{R}^n$.
\item 
${\cal T}$ is {\it ergodic}, i.e. any set $E \in\Xi$ such that
$\displaystyle \mu \big( ({\cal T} (x) E \cup E) \setminus ({\cal T} (x) E \cap E) \big) =0, $ $\forall x\in \mathbb{R}^n$ satisfies either $\mu (E)=0$ or $\mu (E)=1.$
\end{enumerate}
With the measurable dynamics introduced above
we associate a $n$-parameters group of unitary operators on
$L^2(\Omega) \equiv L^2(\Omega,\Xi,\mu )$, as follows
$$
(U(x)f)(\omega)=f({\cal T}(x)\omega),\qquad f\in L^1(\Omega).
$$
The map $x\to U(x)$ is continuous in the strong operator topology, see for instance Jikov et al \cite{JKO94}.


%
Next, let $\mathcal{F}\in\Xi$ be
such that $\mu(\mathcal{F})>0$ and $\mu(\Omega\setminus\mathcal{F})>0$. Starting from the set $\mathcal{F}$, we
 define a random pore structure
$F(\omega)\subset\mathbb{R}^n$, $\omega\in\Omega$. \ It is obtained from $\mathcal{F}$
by setting
\begin{equation}\label{f_set}
F(\omega) = \{x\in\mathbb{R}^n\, :\, {\cal T}(x)\omega\in \mathcal{F}\}.
\end{equation}
%
In what follows we suppose that $F(\omega)$ is
open and connected a.s. (for almost all $\omega\in\Omega$). In a complementary way, the random rigid skeleton structure $M(\omega)$ is introduced
  by setting
\begin{equation}\label{m_set}
{\cal M}=\Omega\setminus\mathcal{F},\qquad
M(\omega)=\mathbb{R}^n\setminus F(\omega).
\end{equation}
We assume that a.s. $ M(\omega)$
is a disperse medium, i.e. a union of mutually disjoint components, called grains, satisfying  the following conditions:
\begin{itemize}
  \item[\bf R1.]  {A.s. $M(\omega)$ is a union of non-intersecting $C^2$-smooth bounded domains.}
  \item[\bf R2.]  The curvature of the boundary of these domains admits a deterministic upper bound.
  \item[\bf R3.]   The distance between any two domains is greater than a positive deterministic constant.
\item[\bf R4.]   The diameter of  any domain is not greater than a positive deterministic constant.
 \item[\bf R5.]   There exists $r_0>0$ such that a.s. any ball of radius $r_0$ in $\mathbb R^n$ has a nontrivial intersection with $M(\omega)$.
\end{itemize}
Under condition {\bf R1.} the connected components of $M(\omega)$ can be enumerated, so that
$$
M(\omega)=\bigcup_{j=1}^\infty M_j(\omega).
$$
In connection with the random set $M(\omega)$ we introduce a homothetic
structure $M_\varepsilon(\omega)$, $\omega\in\Omega$, by
\begin{equation}\label{matrix}
M_\varepsilon(\omega)=\{x\in\mathbb{R}^n\,:\,\varepsilon^{-1}x\in
M(\omega)\};
\end{equation}
For more details on the homogenization of linear PDEs in perforated random domains we refer to the monograph Jikov et al \cite{JKO94}.

\medskip
Here we provide several examples of random disperse media.\\[2mm]
$\boldsymbol{\mathcal{E}1}.$  {\sl Random perturbation of periodic structure}.
Let $\zeta^j$, $j\in\mathbb Z^n$, be a collection of  independent and identically distributed (i.i.d.) random vectors
in $\mathbb R^n$ taking on values in $[ -\frac14,\frac14]^n$. Denote by $r^j$,
$j\in\mathbb Z^n$, a collection of i.i.d. random variables such that $\frac13\leq r^j\leq \frac23$, and by $\eta$ a random vector
that is uniformly distributed on the cube $[-\frac12,\frac12]^n$ and independent  with respect to $\zeta^j$ and $r^j$.  Taking for $M(\omega)$
the union of balls centered at $j+\zeta^j+\eta$ of radius $r^j$ we obtain an example of  random statistically homogeneous disperse medium. In this case $\mathcal{M}=\{\omega\,:\, 0\in M(\omega)\}$. \\[1.5mm]
$\boldsymbol{\mathcal{E}2}.$  {\sl Bernoulli spherical structure}.   Consider a collection of i.i.d. random variables $\xi^j$,
$j\in\mathbb Z^d$,  with values $0$ and $1$, and an independent vector $\eta$ uniformly distributed on  $[-\frac12,\frac12]^n$.
We say that a vertex $j\in\mathbb Z^n$ is open if $\xi^j=1$, and closed if  $\xi^j=0$. We then define $M(\omega)$  by 
$$
M(\omega)=\bigcup\limits_{j\hbox{ is open}}\big\{x\in\mathbb R^n\,:\,|x-j-\eta|\leq\frac12\big\}\,\cup
\bigcup\limits_{j\hbox{ is closed}}\big\{x\in\mathbb R^n\,:\,|x-j-\eta|\leq\frac14\big\}.
$$
Then we set $\mathcal{M}=\{\omega\,:\, 0\in M(\omega)\}$.
\\[1.5mm]
$\boldsymbol{\mathcal{E}3}.$  {\sl Poisson spherical structure}.  Let $\mathcal{P}$ be a Poisson process in $\mathbb R^n$ with
intensity one, 
the averaged density of points is equal to one.
By definition $\mathcal{P}$ is a random locally finite subset of $\mathbb R^n$ such that for any bounded Borel set
$G\subset\mathbb R^n$ the number of points in  $\mathcal{P}\cap G$ has a Poisson distribution with intensity $|G|$, and
for any finite collection of  bounded disjoint Borel sets $G_1,\ldots,G_N$ the random variables defined as the number of points of
$\mathcal{P}$ in each of these sets are independent.  \\
Denote the points of $\mathcal{P}$ by $z_k$, and the cells of the Voronoi tessellation generated by $\{z_k\}$ by $\Xi_k$.
Given $r>0$ we choose those $z_k$ for which the ball of radius $r$ centered at $z_k$ belongs to $\Xi_k$, and denote the union of such balls by $M(\omega)$.  Again, $\mathcal{M}=\{\omega\,:\, 0\in M(\omega)\}$.

\medskip
By construction the random domains $M(\omega)$ defined in the first two examples are stationary and satisfy conditions
{[\bf R1.]--[\bf R5.]} that is in both examples $M(\cdot)$ is a random disperse medium.   To analyse the ergodic properties
of these media we consider for the sake of definiteness the second example and introduce a probability space by setting $\Omega=\{0,1\}^{\mathbb Z^n}$ and taking the cylindrical $\sigma$-algebra $\Xi$ and the product measure $\mu$.
It suffices to check the ergodicity for the integer shifts only.
We then define $\mathcal{T}$ by $\mathcal{T}_z\omega(\cdot)=\omega(\cdot+z)$, $z\in\mathbb Z^n$.
 It is straightforward to show that this dynamical system is ergodic.  Indeed, if there exists a measurable invariant set
 $A$ such that $\mathcal{T}_z(A)=A$, then for any $\delta>0$ there exists a set $A_\delta$ supported by a finite number of $z\in\mathbb Z^n$ such that  $\mu(A\boldsymbol{\scriptstyle \Delta} A_\delta)\leq\delta$; here $\boldsymbol{\scriptstyle \Delta}$ stands for the symmetric difference. Since $\mathcal{T}$ preserves measure $\mu$,
$\mu(A\boldsymbol{\scriptstyle \Delta} \mathcal{T}_z(A_\delta))\leq\delta$ for any $z\in\mathbb Z^n$.   For sufficiently large $z_0$ we have
$\mu(A\cap \mathcal{T}_{z_0}(A))=\mu(A)$ and
$\mu{(A_\delta\cap\mathcal{T}_{z_0}(A_\delta))}=(\mu(A_\delta))^2$. Therefore,  $\mu(A)$ is equal to either $0$ or $1$.\\
 The ergodicity of the medium in the first example can be checked in the same way.

The random medium introduced in the third example is stationary, ergodic and satisfies conditions {[\bf R1.]--[\bf R4.]};
however, condition {[\bf R5.]} fails to hold. Filling large empty areas with a regular grid of balls one can rearrange this random medium to make condition  {[\bf R5.]} also hold.

\medskip
Notice that in all the above examples the balls can be replaced with smooth bounded domains whose geometry might be
random and should satisfy a number of natural conditions. 
 For more examples see Bourgeat et al \cite{BAP03}.

\medskip
Let $G$ be a smooth bounded domain in $\mathbb{R}^n$. After having
chosen the random structure in $\mathbb{R}^n$, we set
%
\begin{equation}\label{eps_dom}
G_1^\varepsilon=\{x\in G\,:\,{\rm dist}(x,\partial G)
\geq \varepsilon\}.
\end{equation}
The random fluid filled pore
system in $G$ is given by
\begin{equation}\label{perf_dom}
G_f^\varepsilon(\omega) = G\setminus(\overline{\bigcup_{j\in \mathcal{J}(\varepsilon)}\varepsilon M_j(\omega)}),
\end{equation}
where
$$
\mathcal{J}(\varepsilon)=\{j\in \mathbb Z^+\,:\, \varepsilon M_j(\omega)\subset G_1^\varepsilon\}.
$$
Then, the random rigid solid skeleton part of $G$ is defined as the complement of
$G_f^\varepsilon(\omega)$ in $G$:
\begin{equation}\label{matr_dom}
G_m^\varepsilon(\omega) = G\setminus\overline{G_f^\varepsilon(\omega)}.
\end{equation}

\medskip

Before giving the convergence results, we recall the definition and some
properties of the stochastic two-scale convergence in the mean
(see Bourgeat et al \cite{BMW} for more details).

Let $D_j$ denote the infinitesimal generator in $L^2(\Omega)$ of
the one-parameter
group of translations in $x_j$, with the other coordinates held equal to zero.  ${\cal D}_j$ is its respective
domain of definition
in $L^2(\Omega)$, i.e. for $f\in{\cal D}_j$
\begin{equation}\label{gener}
(D_j f)(\omega) = \frac{\partial}{\partial x_j } (U(x)f)
(\omega)\vert_{x=0}
\end{equation}
Then $\{\sqrt{-1}\,D_j,\ j=1,\dots,n\}$ are closed, densely--defined
and self--adjoint operators which commute pairwise on ${\cal D}(\Omega)=
\bigcap\limits_{j=1}^n {\cal D}_j$.
${\cal D}(\Omega)$ is a Hilbert space with respect to the scalar product
\begin{equation}\label{in_prod}
(f,g)_{{\cal D}(\Omega)} = (f,g)_{L^2(\Omega)} +
\sum\limits_{j=1}^n (D_j f,D_j g)_{L^2(\Omega)}
\end{equation}

After (\ref{gener}),  the stochastic gradient
$\{\nabla_\omega f\}$, divergence
$\{{\rm div}_\omega f\}$ and curl $\{{\rm curl}_\omega f\}$,
read as follows
\begin{equation}\label{stoch_obj}
\left\{ \begin{array}{ll}
\nabla_\omega f & = (D_1 f,\dots,D_n f)
\\
\displaystyle
{\rm div}_\omega g & = \sum\limits_j D_j g_j
\\
{\rm curl}_\omega g & = (D_i g_j-D_j g_i),\ i\neq j, \ i,j=1, \dots, n
\end{array} \right.
\end{equation}
In addition, we  use the following spaces:
\begin{equation}\label{l2_pot}
{\cal V}^2_{\rm pot}(\Omega) = \{f\in L^2_{\rm pot}(\Omega),\,\eit\{f\}
= 0\}
\end{equation}
\begin{equation}\label{l2_sol}
{\cal V}^2_{\rm sol}(\Omega) = \{f\in L^2_{\rm sol}(\Omega),\,\eit\{f\}
= 0\}
\end{equation}
where $L^2_{\rm pot}(\Omega)$ (respectively $L^2_{\rm
sol}(\Omega))$ is the set of all $f\in(L^2(\Omega))^n$ such that
almost all realizations $f({\cal T} (x)\omega)$ are potential
(resp. solenoidal) in $\mathbb{R}^n$; for more details see
Jikov et al \cite {JKO94}. 
Notice that $\displaystyle L^2_{\rm
sol}(\Omega))= L^2_{\rm pot}(\Omega)^{\bot}$.

%
\begin{remark}\label{rem_3}
In the case of disperse media the
functions from ${\cal D}(\Omega)$ vanishing on ${\cal M}$ are
dense in $L^2({\cal F})$, see \cite[Remark 2.4]{BAP03}. 
In the terminology of \cite{wright}, $\mathcal{F}$ is called $\mathcal{T}$-open in $\Omega$. Furthermore, due to connectivity of $F(\omega)$,
$\nabla_\omega \psi =0$ in $\mathcal{F}$ implies that $\psi$ does not depend on $\omega$ in $\mathcal{F}$. Hypothesis {\bf R1.-R3.} are sufficient to make the assumptions on $\mathcal{F}$  in Proposition \ref{twosccompact} below satisfied.
\end{remark}

Next, following \cite{BMW}, we say that an element $\psi\in L^2(G\times\Omega)$ is
admissible if the function
$$\psi_{\cal T}\,:\,(x,\omega) \longrightarrow \psi(x,{\cal T}(x)\omega),
\qquad (x,\omega)\in G\times\Omega,
$$
defines an element of $L^2(G\times\Omega)$.

\noindent
Examples
 of admissible two-scale functions are elements from
$C(\overline{G}\,;\,L^\infty(\Omega))$, $L^2(G\,;\,B(\Omega))$
and finite linear combination of functions of the form
$$
(x,\omega)\longrightarrow f(x)g(\omega),\quad (x,\omega)\in
G\times\Omega,\ f\in L^2(G),\ g\in L^2(\Omega),
$$
(see Bourgeat et al \cite{BMW}).

The notion of stochastic two--scale convergence
in the mean was introduced in Bourgeat et al  \cite{BMW}. It generalizes the two-scale convergence in the periodic setting introduced by Nguetseng in \cite{NGU} and Allaire in \cite{All92}. We recall it for comfort of the reader

\begin{definition}
A bounded sequence $\{ u^\varepsilon\}$ of functions from $L^2(G\times\Omega)$
is said to converge stochastically two-scale in the mean (s.2-s.m.) towards
$u\in L^2(G\times\Omega)$
if for any admissible $\psi\in L^2(G\times\Omega)$ we have
\begin{equation}\label{def_2s}
\lim\limits_{\varepsilon\to 0}\,\int\limits_{G\times\Omega}
u^\varepsilon(x,\omega )\psi(x,{\cal T}(\frac{x}{\varepsilon})\omega)dxd\mu
=\int\limits_{G\times\Omega} u(x,\omega)\psi(x,\omega)dxd\mu.
\end{equation}
\end{definition}



\noindent
Our functions are defined on $G_f^\varepsilon(\omega)$ and not on  $G$. It is the well-known complication appearing in homogenization of   Neumann
problem in perforated domains. Motivated by  Jikov et al \cite{JKO94} and Bourgeat et al \cite{BAP98}, \cite[Proposition 2.2]{BAP03} we have the following $2-$scale compactness result.

\begin{proposition}\label{twosccompact}
Let $\{ \Phi^\varepsilon\}\subset H^1(G)$ and $\{ \Psi^\varepsilon\}\subset H^1(G)$ be such  sequences that
\begin{gather}
\left\{\begin{array}{ll}
\Vert \Phi^\varepsilon\Vert_{L^2(G)} +
\Vert\nabla \Phi^\varepsilon\Vert_{L^2(G^\varepsilon_f(\omega))}  \leq C, \label{est_glob1} \\[2mm]
\Vert \Psi^\varepsilon\Vert_{L^2(G)} +\varepsilon
\Vert\nabla \Psi^\varepsilon\Vert_{L^2(G^\varepsilon_f(\omega))}  \leq  {C},
\end{array} \right.
\end{gather}
{
and assume that assumptions {\bf R1.}--{\bf R4.} are fulfilled.
%

%
Then there exist functions $\Phi_0 \in H^1(G)$, $\Psi_1 \in L^2(G;{\cal D}(\Omega))$,
$\Psi_1=0$ on $\mathcal{M}$,
and $\Phi_1\in L^2(G;\,X)$, $\Phi_1=0$ on ${\cal M}$,
such that, up to a subsequence,}
\begin{eqnarray}\label{two_s1}
\chi_{G_f^\varepsilon(\omega)} \Phi^\varepsilon&\mathop{\longrightarrow}\limits^{\hbox{s.2-s.m.}}&
 \chi_{\mathcal{F}} (\omega) \Phi_0 (x),
\\
\label{two_s12}
\chi_{G_f^\varepsilon(\omega)} \Psi^\varepsilon&\mathop{\longrightarrow}\limits^{\hbox{s.2-s.m.}}&
 \Psi_1 (x, \omega),
\\
\label{two_s2}
\chi_{G_f^\varepsilon(\omega)}\nabla \Phi^\varepsilon
&\mathop{\longrightarrow}\limits^{\hbox{s.2-s.m.}}&
\chi_{\Omega\setminus{\cal M}}[\nabla_x \Phi_0 (x)+\Phi_1(x,\omega)]
\\
\label{two_s3}
\varepsilon\chi_{G_f^\varepsilon(\omega)}\nabla \Psi^\varepsilon
&\mathop{\longrightarrow}\limits^{\hbox{s.2-s.m.}}&
\chi_{\mathcal{F}}(\omega)\nabla_\omega \Psi_1 (x,\omega)
\end{eqnarray}
\end{proposition}
\begin{proof}
{
As it was shown in \cite{JKO94},  \cite{wright} and \cite{ZhPi}, under assumptions {\bf R1.}--{\bf R4.} the set of all
functions $\psi\in{\cal D}(\Omega)$
such that $\psi=0$ on ${\cal M}$, is dense in $L^2(\Omega\setminus
{\cal M})$, and whenever $\nabla_\omega \psi =0$ in $\mathcal{F}$, then $\psi$ does not depend on $\omega$ in $\mathcal{F}$.
}

Let $X$ be the closure of the space ${\cal V}^2_{\rm pot}(\Omega)$
in $L^2(\Omega\setminus{\cal M})^n$. Under  that same assumptions  {\bf R1.}--{\bf R4.}
the tensor ${\cal A}_N^0$ associated to the homogenized Neumann
problem and defined by
\begin{equation}\label{N_matr}
\xi\cdot{\cal A}^0_N\xi = \inf\limits_{v\in X}\int\limits_{\Omega\setminus
{\cal M}}|\xi+v|^2 d\mu,\qquad \xi\in\rit^n,
\end{equation}
is positive definite.
Using the above a priori estimates and the stochastic two-scale in the mean
compactness theorem
from Bourgeat et al \cite{BMW}, we conclude that, after
taking a proper subsequence, the sequences
$\{\Phi^\varepsilon\}$, $\{\chi_{G_f^\varepsilon(\omega)}\nabla
\Phi^\varepsilon\}$,  $\{\Psi^\varepsilon\}$ and
$\{\varepsilon\chi_{G_f^\varepsilon(\omega)}\nabla \Psi^\varepsilon\}$
have stochastic two-scale limits.
We have then:
\begin{itemize}
\item
$\Phi^\varepsilon
\mathop{\longrightarrow}\limits^{\hbox{s.2-s.m.}}
\Phi_0(x,\omega)$
\item
$\chi_{G_f^\varepsilon(\omega)}\nabla \Phi^\varepsilon
\mathop{\longrightarrow}\limits^{\hbox{s.2-s.m.}}\xi_0(x,\omega)$
\item
$\Psi^\varepsilon
\mathop{\longrightarrow}\limits^{\hbox{s.2-s.m.}}
\Psi_1 (x,\omega)$
\item
$\varepsilon \chi_{G_m^\varepsilon(\omega)}\nabla \Psi^\varepsilon
\mathop{\longrightarrow}\limits^{\hbox{s.2-s.m.}}z_0(x,\omega)$
\end{itemize}
We should find relations between $\Phi_0$, $\xi_0$, $v$ and $z_0$.

Concerning the relation between $\xi_0$ and $\Phi_0$, it was considered in \cite[Proposition 2.1]{BAP03},  and it was proved that
$$ \xi_0 =0 \quad \mbox{on} \quad G\times \mathcal{M}\quad \mbox{and} \quad \xi_0(x,\omega)-\nabla_x \Phi_0(x)\in L^2(G;
X)^n. $$

It remains to identify $z_0$ .

Taking into account the ergodicity of the dynamical system and
connectivity of the fractures, by the same arguments as in
Bourgeat et al \cite{BMW}, Bourgeat et al \cite{BAP03} and  Wright \cite{wright}, we conclude
$$ z_0 (x, \omega)=0 \quad \mbox{on} \quad G\times \mathcal{M}\quad \mbox{and} \quad z_0(x,\omega)=\chi_{\mathcal{F}} \Psi_1 (x, \omega)\in \mathcal{D} (\Omega). $$

\end{proof}
\begin{remark}
It should be noted that ${\cal A}_N^0$ is always positive definite
in the periodic case if the fracture part is connected.
\end{remark}


\subsection{Linearization}\label{2.2}

In this subsection we follow the lead of O'Brien \& White \cite{OBW:78} and proceed with semi-linearization of  system (\ref{EPPR1})-(\ref{EPPR5}).  The 
 static electric potential $\Psi^{ext}(x)$ 
   and the applied fluid force $\mathbf{f}^*(x)$
are assumed to be sufficiently small.
No smallness condition is imposed on $N_\sigma \sigma$ and the Poisson-Boltzmann equation (\ref{EPPR3b}) remains non-linear.

After O'Brien \& White \cite{OBW:78},
we write the electrokinetic unknowns as
\begin{gather}
n_i^\ep (x) = n^{0,\ep}_i (x) + \delta n_i^\ep (x), \quad  \Psi^\ep (x) = \Psi^{0, \ep} (x) + \delta \Psi^\ep (x), \notag \\
\mathbf{u}^\ep(x) = \mathbf{u}^{0, \ep}(x) + \delta\mathbf{u}^\ep(x), \quad  p^\ep (x) = p^{0, \ep} (x) + \delta p^\ep (x), \notag
\end{gather}
where $n^{0,\ep}_i, \Psi^{0,\ep}, \mathbf{u}^{0,\ep}, p^{0,\ep}$ are the equilibrium quantities, corresponding
to $\mathbf{f}^*=0$ and $\Psi^{ext }=0$. The $\delta$ prefix indicates the size of  perturbation.

At zero order, corresponding to  $\mathbf{f}^*=0$ and $\Psi^{ext}=0$, we search for an equilibrium solution of the form
\begin{gather}
\mathbf{u}^{0, \ep} = 0 \, , \quad
p^{0,\ep} =  \sum_{j=1}^N n_j^c \exp \{-z_j \Psi^{0,\ep} \} \, ,\notag \\
n_j^{0, \ep} (x) = n_j^c \exp \{ -  z_j \Psi^{0, \ep} (x) \} \, , \label{NJ0}
\end{gather}
where $\Psi^{0, \ep }$ solves the Poisson-Boltzmann equation
\begin{equation}
\label{BP1}
\left\{ \begin{array}{l}
       \dsp -\ep^2 \Delta \Psi^{0, \ep } =\beta  \sum_{j=1}^N z_j n_j^c e^{ -  z_j \Psi^{0, \ep }  } \ \mbox{ in } \  G_f^\varepsilon(\omega) , \\
       \dsp  \ep \nabla \Psi^{0, \ep } \cdot \nu = -N_\sigma \sigma \ \mbox{ on } \,  \partial G_f^\varepsilon(\omega)\setminus \p G, \quad \Psi^{0, \ep }  =0 \ \mbox{ on } \,  \partial  G.
     \end{array}
   \right.
\end{equation}
Furthermore,  we assume that all valences $z_j$ are different and
\begin{equation}
\label{valence}
z_1 < z_2 < ... < z_N , \quad z_1<0<z_N .
\end{equation}
We denote by $j^+$ and $j^-$ the sets of positive and negative valencies.

We note that problem (\ref{BP1}) is equivalent to the following minimization problem:
\begin{equation}\label{Minpsi}
   \inf_{\varphi \in W_\ep} J_\ep(\varphi) ,
\end{equation}
with $W_\ep=\{ z\in  H^1 (G_f^\ep (\omega) ) \ | \ z=0 \; \mbox{ on } \; \p G \} $
and
\begin{gather*}
J_\ep(\varphi) = \frac{\ep^2}{2} \int_{ G_f^\varepsilon(\omega)} | \nabla \varphi  |^2 \ dx
+ \beta \sum_{j=1}^N \int_{\partial G_f^\varepsilon(\omega)}  n_j^c   e^{ -z_j \varphi  } \  dx
    + \ep N_\sigma \int_{\partial G_f^\varepsilon(\omega)} \sigma \varphi \ dS .
\end{gather*}
The functional $J_\ep$ is strictly convex, which gives the uniqueness of the minimizer. Nevertheless, for arbitrary non-negative $\beta , n_j^c$ and $N_\sigma$, $J_\ep$ may be {\bf not coercive} on $W_\ep$ if all $z_j$'s have the same sign.
It suffices to take as $\varphi$  constants of the same sign as $z_1$ and $z_N$'s and tending to infinity. 
  Following the literature, this degeneracy is handled by imposing the {\bf bulk electroneutrality condition}
\begin{equation}
\label{Neutrality}
   \sum_{j=1}^N z_j n_j^c =0,
\end{equation}
which guarantees that for $\sigma=0$, the unique solution is $\Psi^{0,\ep} =0$.

The second difficulty is that $J_\ep$ is not defined on $W_\ep$, but rather on $W_\ep \cap L^\infty (G_f^\ep (\omega))$.
\begin{remark}
The bulk electroneutrality condition (\ref{Neutrality}) is not
a restriction. Actually, all our results hold under
the much weaker assumption that all valences $z_j$
do not have the same sign. We refer to \cite{AllBrizzDufMiPi:12} for the argument how to reduce the general case to
(\ref{Neutrality}).
\end{remark}

\begin{remark}
\label{lem.pb}
Assume that the electroneutrality condition (\ref{Neutrality}) holds true
and $\sigma$  be a smooth bounded function.
Then, in the deterministic setting, it was  proved in Allaire et al \cite{beta} that problem (\ref{Minpsi}) has a unique solution $\Psi^{0,\ep} \in W_\ep\cap L^\infty (G_f^\ep (\omega))$.
\end{remark}

Motivated by  the computation of $n_i^{0, \ep}$, having the form of the Boltzmann equilibrium distribution,  we follow again lead of
\cite{OBW:78} and introduce the so-called ionic potential
$\Phi_i^\ep$ which is defined in terms of $n_i^\ep$ by
\begin{equation}
\label{BP2}
n_i^\ep = n_i^c \exp \{ -  z_i ( \Psi^\ep  + \Phi_i^\ep + \Psi^{ext} ) \} .
\end{equation}
After linearization (\ref{BP2}) leads to
\begin{equation}
\label{BP3}
\delta n_i^\ep (x) = - z_i n_i^{0,\ep} (x) (\delta \Psi^\ep (x) + \Phi_i^\ep (x) + \Psi^{ext}(x) ).
\end{equation}
Introducing (\ref{BP3}) into (\ref{EPPR1})-(\ref{Nernst1}) and linearizing yields
the following system
\begin{gather}
\ep^2 \Delta \mathbf{u}^\ep - \nabla P^\ep = \mathbf{f}^* -
\sum_{j=1}^N z_j n_j^{0,\ep}(x) (\nabla\Phi_j^\ep + \mathbf{E}^* ) \quad \mbox{ in } \;  G_f^\varepsilon(\omega) , \label{LIN2} \\
\div \mathbf{u}^\ep =0 \ \mbox{ in }  \;  G_f^\varepsilon(\omega) , \quad
\mathbf{u}^\ep =0 \ \mbox{ on } \; \partial G_f^\varepsilon(\omega) ,\label{LIN4} \\
\div \left( n^{0, \ep}_j (x) \big( \nabla \Phi_j^\ep + \mathbf{E}^* + \frac{\pej}{z_j} \mathbf{u}^\ep \big)\right) =0 \ \mbox{ in } \;  G_f^\varepsilon(\omega) ,\; j=1, \dots , N,
\label{EPPR4L} \\
(\nabla \Phi_j^\ep + \mathbf{E}^*) \cdot \nu = 0 \; \mbox{ on } \, \partial G_f^\varepsilon(\omega)\setminus \p G, \quad  \Phi_j^\ep =0  \; \mbox{ on } \, \partial G, \; j=1, \dots , N,\label{EPPR4a}
\end{gather}
where the perturbed velocity is actually equal to the overall velocity and, for convenience,
we introduced a global pressure $P^\ep$
\begin{equation}\label{LIN5}
\delta \mathbf{u}^\ep = \mathbf{u}^\ep \, , \quad
P^\ep = \delta p^\ep +  \sum_{j=1}^N
z_j n_j^{0, \ep} \left( \delta \Psi^\ep + \Phi_j^\ep  + \Psi^{ext , *} \right) .
\end{equation}
For the choice of the boundary conditions on $\p G$, we have followed O'Brien \& White \cite{OBW:78}.
It is important to remark that, after the global pressure $P^\ep$ has been introduced,
$\delta \Psi^\ep $ does not enter equations (\ref{LIN2})-(\ref{EPPR4a})
and thus is decoupled from the main unknowns $\mathbf{u}^\ep$, $P^\ep$ and $\Phi_i^\ep$.

\subsection{Poisson-Boltzmann equation in the random geometry}\label{PerioPB}

%

Rescaling of
the Poisson-Boltzmann equation in \eqref{BP1} yields its form valid in $F(\omega):$
\begin{equation}
\label{BP1per}
\left\{ \begin{array}{ll}
\dsp -\Delta_y \Psi^{0} = \beta \sum_{j=1}^N z_j n_j^c e^{ - z_j \Psi^{0} } =- \beta n_H (\Psi^0 )\ \mbox{ in } \  F (\omega), &  \\
\dsp   \nabla_y \Psi^{0} \cdot \nu = -N_\sigma \sigma (\mathcal{T}(y)\omega ) \ \mbox{ on } \, \p F (\omega).&
\end{array} \right.
\end{equation}
Problem (\ref{BP1per}) does not have boundary conditions at infinity.
They are hidden in the statistical  homogeneity of a solution. We recall that in the periodic case, one can search for globally bounded smooth solution in the
space and it turns out that they are necessarily the periodic ones.

{
We will show below  that problem \eqref{BP1per} is well-posed and has a unique solution and that the solution
$\Psi^{0,\ep}$ of problem \eqref{BP1} stochastically two-scale converges as $\ep\to0$ to the solution $\Psi^0(\cdot)$
of \eqref{BP1per}.

Clearly, the function $\Psi^0(\mathcal{T}(x/\ep))$ satisfies a.s. the
{
Poisson-Boltzmann equation and Neumann's boundary condition } in \eqref{BP1}. {
Since we are interested in the effective bulk behavior of the potential,  $\Psi^0(\mathcal{T}(x/\ep))$
is in fact the desired approximation to be used in the concentration coefficients $n_j^{0, \varepsilon}$ of the equations (\ref{LIN2}) and (\ref{EPPR4L}).
}

\vskip5pt

In order to derive $L^\infty$-bounds for problem (\ref{BP1}),
we first handle the non-homogeneous Neumann condition and
study the following 
$\ep$-problem
\begin{equation}
\label{BP1aux}
\left\{ \begin{array}{ll}
\dsp -\Delta_y V^\ep + V^\ep =0\ \mbox{ in } \  G_f (\omega) = \frac{1}{\ep} \ G_f^\varepsilon(\omega), &  \\
\dsp   \nabla_y V^\ep \cdot \nu = - N_\sigma \sigma (y ) \ \mbox{ on } \,  \frac{1}{\ep} \ \p G_m^\varepsilon(\omega)= \bigcup_{j\in \mathcal{J}(\varepsilon)}\p M_j(\omega), &\\
V^\ep =0  \ \mbox{ on } \ \frac{1}{\ep} \ \p G.&
\end{array} \right.
\end{equation}

\begin{proposition}\label{prop7}
Let $\sigma $ be a bounded function such that  $\|\sigma\|_{L^\infty(\partial F(\omega))}\leq C_0$ a.s.
and assume that  conditions {\bf R1.}--{\bf R3.} are fulfilled.
Then a.s. there exist constants $V_m $ and $V_M $, independent of $\ep$, such that
\begin{equation}\label{BoundsV}
  V_m \leq V^\ep (y, \omega)) \leq V_M\quad \mbox{a.e. on} \quad F (\omega).
\end{equation}
\end{proposition}
\begin{proof}

  First, we recall that problem  (\ref{BP1aux}) has a unique solution $V^\ep \in H^1 (G_f (\omega)$, $V^\ep =0$ on $\p G/\ep$. We search a $L^\infty$-bound independent of $\ep$.

Under conditions {\bf R1.}--{\bf R3.}, results from Gilbarg \& Trudinger \cite[Appendix 14.6]{GT} yield that the distance $d$ is an element of $C^2 (\Gamma_k)$, with $\Gamma_k =\{ \ x\in F (\omega) \ | \ d(x)=\mbox{dist} (x, \p F (\omega)) \leq 2 k \}$ and $1/k$ bounds the positive curvature of $\p F (\omega)$.

Let $h\in C^2 [0, +\infty)$ be a nonnegative function such that $h(t)=0$ for $t\geq k$ and $\displaystyle h' (0) =N_\sigma ||\sigma ||_{L^\infty (\p F(\omega))} =C_0$. Then there is a constant ${\hat C}$ , independent of $R$, such that function ${\hat a}_\ep =h(d)$ satisfies
$$
0\leq \hat a_\ep\leq \hat C\quad\hbox{in }\ F(\omega),\qquad |\Delta \hat a_\ep |< \hat C\quad
\hbox{in }\ F(\omega),\qquad
\nabla_y \hat a_\ep \cdot \nu =h'(0) = C_0  \ \mbox{ on } \,  \bigcup_{j\in \mathcal{J}(\varepsilon)} \p M_j(\omega).
$$

Furthermore
$$-\Delta (\hat a_\ep+\hat C) +(\hat a_\ep+\hat C)>0$$
and $\displaystyle \frac{\p }{\p \nu} \hat a_\ep \geq N_\sigma \sigma$.
By the maximum principle
$$
 V^\ep\leq \hat a_\ep+\hat C\leq 2\hat C.
$$
In the same way we can show that  $-2\hat C\leq V^\ep\leq 2\hat C$ and
 the solution $V^\ep$ of problem \eqref{BP1aux} satisfies  \eqref{BoundsV}.
\end{proof}
\begin{remark} \label{rem8} Let us suppose, in addition, that $\sigma $ is bounded in $C^1 (\p F(\omega))$.  The function $V^\ep$, solving problem \eqref{BP1aux}, is in fact the solution for the variational problem
\begin{gather}
  \mbox{Find} \;  V^\ep \in H^1 (G_f (\omega)),\;  V^\ep =0 \; \mbox{on} \; \frac{1}{\ep} \  \p G \; \mbox{such that} \notag \\
  \hskip-8pt \int_{G_f (\omega)} (\nabla_y V^\ep \cdot \nabla_y \varphi +   V^\ep  \varphi) \ dy = - N_\sigma \bigcup_{j\in \mathcal{J}(\varepsilon)} \int_{\p M_j (\omega)} \sigma \varphi \ dS_y = -N_\sigma \bigcup_{j\in \mathcal{J}(\varepsilon)} \int_{\p M_j (\omega)} \sigma \nabla d \cdot \nu \varphi \ dS_y  =\notag \\
  \hskip-8pt - N_\sigma \int_{ G_f (\omega)} \big( \frac{\sigma}{C_0} h(d) \ \nabla_y d \cdot \nabla_y \varphi + \frac{\sigma}{C_0} h(d) \  \Delta d \ \varphi + \frac{\sigma}{C_0} h'(d) \ |\nabla_y d|^2 \varphi + \frac{\nabla_y \sigma}{C_0} \cdot \nabla_y d \ h(d) \varphi \big) \ dy. \label{Variaaux}
\end{gather}
In addition to estimate \eqref{BoundsV}, we have
\begin{equation}\label{Estvaraux}
  \int_{G_f (\omega)} (|\nabla_y V^\ep |^2 + | V^\ep |^2 )\ dy \leq C | G_f (\omega)|, \quad \mbox{(a.s.)} \; \mbox{ in } \; \omega.
\end{equation}
Estimates (\ref{BoundsV}) and (\ref{Estvaraux}) allow passing to the limit $\ep \to 0$ of the sequence $\{ V^\ep \}$.
\end{remark}
Having $L^\infty$-bounds for the solution of auxiliary problem (\ref{BP1aux}) allows proving existence of a bounded solution to problem
(\ref{BP1per}).
\vskip1pt
We start with the $\ep$-problem:
\begin{equation}
\label{BP1monaux}
\left\{ \begin{array}{ll}
\dsp -\Delta_y \Psi^{\ep}  =\beta \sum_{j=1}^N z_j n_j^c e^{ - z_j \Psi^{\ep} } =
- \beta n_H (\Psi^\ep )\ \mbox{ in } \  G_f (\omega), &  \\
\dsp   \nabla_y \Psi^\ep \cdot \nu = - N_\sigma \sigma (y ) \ \mbox{ on } \,  \frac{1}{\ep} \ \p G_m^\varepsilon(\omega), &\\
\Psi^\ep =0  \ \mbox{ on } \ \frac{1}{\ep} \ \p G.&
\end{array} \right.
\end{equation}
\begin{theorem} \label{PoissBoleq} Under the
electroneutrality condition (\ref{Neutrality}) and
hypotheses {\bf R1.}--{\bf R4.}, a.s., there exists a unique solution $\Psi^\ep \in H^1 (G_f (\omega)) \cap L^\infty (G_f (\omega))$ of problem (\ref{BP1monaux}), such that
\begin{gather}
 \int_{G_f (\omega)} (|\nabla_y \Psi^\ep |^2 + | \Psi^\ep |^2 )\ dy \leq C | G_f (\omega)|,
\label{Energyep}   \\
  || \Psi^\ep ||_{L^\infty (G_f (\omega))}\leq C, \label{Boundep}
\end{gather}
where $C$ is a deterministic constant, independent of $\ep$.
\end{theorem}
\begin{proof}$\; $

\vskip1pt

\underline{Step 1}.
\vskip1pt

Let $N > \max \{ |C_1| , |C_2| \}$, where $C_1$ and $C_2$ are upper and lower bounds for $\varphi_{N}$ obtained in Step 2 and independent of $N$. We introduce the cut-off nonlinearity $n_{HN}$ by
\begin{equation}\label{Cutoff1}
  n_{HN} (z) =\left\{
                 \begin{array}{ll}
                   \displaystyle  n_H (z) \quad \mbox{for} \quad |z|\leq N; &  \\
                   \displaystyle n_H (N) +z-N  \quad \mbox{for} \quad z > N ; &  \\
                   \displaystyle n_H (-N) +z+N  \quad \mbox{for} \quad z < -N. &
                 \end{array}
               \right.
\end{equation}
The cut-off functional is  defined by
\begin{gather*}
J_{N}(\varphi) = \frac{1}{2} \int_{ G_f (\omega)} | \nabla_y \varphi  |^2 \ dy
+ \beta  \int_{G_f (\omega)} \Gamma_N ( \varphi ) \  dy
    -   \int_{ G_f (\omega)} ( V^\ep \varphi +\nabla_y V^\ep \cdot \nabla_y \varphi )\ dy ,
\end{gather*}
where
\begin{equation*}
  \Gamma_N (z) =\left\{
                 \begin{array}{ll}
                   \displaystyle  \sum_{j\in j^+ \cup j^-} n_j (z) -\sum_{j\in j^- \cup j^+}  n_j (0)  \quad \mbox{for} \quad |z|\leq N; &  \\
                   \displaystyle \sum_{j\in j^+ \cup j^-}  n_j (N) + n_H (N) (z-N) + \frac{1}{2} (z-N)^2 - \sum_{j\in j^- \cup j^+}  n_j (0) \quad \mbox{for} \quad z > N ; &  \\
                   \displaystyle  \sum_{j\in j^+ \cup j^-}  n_j (-N) + n_H (-N) (z+N) + \frac{1}{2} (z+N)^2 - \sum_{j\in j^- \cup j^+}  n_j (0) \quad \mbox{for} \quad z < -N. &
                 \end{array}
               \right.
\end{equation*}
Then problem
\begin{equation}\label{Minimcut}
  \min_{\varphi \in W} J_{N} (\varphi)
\end{equation}
 where $W= \{ \ \varphi \in H^1 (G_f (\omega)) \ | \;  \varphi=0\ \hbox{on } \p G/ \ep \ \}$, has a unique solution $\varphi_{N} $. Furthermore
\begin{equation}\label{AprioriN1}
  || \nabla_y \varphi_{N} ||_{L^2 (G_f (\omega))^n} + || \varphi_{N} ||_{L^2 (G_f (\omega))} \leq C | G_f (\omega)|,
\end{equation}
where $C$ does not depend on $N$ and $\ep$.

\underline{Step 2}.

Our next goal is to establish $L^\infty$ estimates for $\varphi_{N}$, independent of $N$ and $\ep$.
We begin with the variational problem
\begin{equation}\label{VarPBinft}
   \int\limits_{G_f (\omega)} \nabla_y (\varphi_{N} - V^\ep) \cdot \nabla_y \varphi \ dy +
  \beta \int\limits_{G_f (\omega)}
  n_{HN} (\varphi_{N})
  \varphi \ dy=
   \int\limits_{G_f (\omega)}  V^\ep \varphi \ dy ,
\end{equation}
for all  $\varphi \in H^1 (G_f (\omega))$, $\varphi=0$ on $\p G/ \ep$.
We take $\varphi = (\varphi_{N} - V^\ep +C_m)_-$, where
$$C_m=V_M + \frac{1}{z_N} \log \bigg( -(V_m  + z_1 \sum_{j\in j^-} n_j^c)_-/ (z_N n_N^c) +1 \bigg).$$
 Inserting this particular test function into equation (\ref{VarPBinft})
 and using that $|V^\ep-C_m|\leq N$,
 yield
\begin{gather}
  \int_{G_f (\omega)} |\nabla_y (\varphi_{N} - V^\ep+C_m)_-|^2  \ dy +\beta \int_{G_f (\omega)} 
  \big( n_{HN} (\varphi_{N}) - n_H (V^\ep-C_m)\big) \underbrace{(\varphi_{N} - V^\ep+C_m)_-}_{=0 \; \mbox{ if } \; \varphi_N \geq N}
   \ dy=\notag \\
 \int_{G_f (\omega)}  ( V^\ep -n_H (V^\ep-C_m))(\varphi_{N} - V^\ep+C_m)_- \ dy .\label{VarPBinft2}
\end{gather}
Next
\begin{gather*}
   V^\ep-n_H (V^\ep-C_m) = V^\ep+ \sum_j z_j n_j^c e^{-z_j (V^\ep-C_m)} \geq V_m + z_N n_N^c e^{-z_N (V_M-C_m)} + z_1 \sum_{j\in j^-} n_j^c  \\
  \geq 0 \quad \mbox{for} \quad C_m=V_M + \frac{1}{z_N} \log \bigg( -(V_m  + z_1 \sum_{j\in j^-} n_j^c)_-/ (z_N n_N^c) +1 \bigg)
\end{gather*}
and we conclude that $\varphi_{N} (y) \geq V_m -V_M + \frac{1}{z_N} \log \bigg( -(V_m  + z_1 \sum_{j\in j^-} n_j^c)_-/ (z_N n_N^c) +1 \bigg)>-N$. The upper bound is analogous.


\underline{Step3}.

With a priori bounds (\ref{AprioriN1}) and the uniform $L^\infty$-bounds, there exists a subsequence of $\varphi_{N}$, again denoted by the same subscript, and $\varphi_\ep \in W\cap L^\infty (G_f (\omega))$, such that
\begin{equation}\label{Limits}
  \left\{
    \begin{array}{ll}
      \displaystyle  \nabla \varphi_{N} \rightharpoonup \nabla \varphi_\ep \quad \mbox{weakly in} \quad L^2 (G_f (\omega))^n; &  \\
\displaystyle  \varphi_{N} \rightharpoonup\varphi_\ep \quad \mbox{weakly in} \quad L^2 (G_f (\omega)); &  \\
       \displaystyle \varphi_{N} \rightharpoonup \varphi_\ep \quad \mbox{weak-* in} \quad L^\infty (G_f (\omega)); &  \\
       \displaystyle \varphi_{N} \to \varphi_\ep \quad \mbox{strongly in} \quad L^2 (G_f (\omega)), &
    \end{array}
  \right.
\end{equation}
as $N\to +\infty$.
Next
$$ \lim \inf_{ N\to +\infty} J_{N} (\varphi_{N}) =\lim \inf_{ N\to +\infty} J (\varphi_{N})\geq J (\varphi_\ep)$$
and
$$ J(g) = J_{{N(g)}} (g) \geq J_{N(g)} (\varphi_{N(g)}) \geq J (\varphi_\ep), \quad \forall g\in W
, \; \Gamma (g) \in L^1 (G_f (\omega)).$$
Hence $\varphi_\ep$ solves variational problem (\ref{BP1monaux}) and provides the minimum in the corresponding minimization problem. The strict convexity implies uniqueness and $\varphi_\ep =\Psi^\ep$.

\end{proof}

It remains passing to the limit in problem (\ref{BP1monaux})  as $\ep\to0$.


{ We are interested in homogenization of a problem posed in $G_f^\ep (\omega)$. We set
\begin{equation}\label{resc1}
  \Psi^{0, \ep} (x)=  \Psi^\ep (\frac{x}{\ep}), \quad x\in G.
\end{equation}
Estimates
(\ref{Energyep})-(\ref{Boundep}) then read
\begin{gather}
 \int_{G_f^\ep (\omega)} (|\ep \nabla \Psi^{0,\ep} |^2 + | \Psi^{0,\ep} |^2 )\ dx \leq C,
\label{EnergyepS}   \\
  || \Psi^{0,\ep} ||_{L^\infty (G_f^\ep (\omega))}\leq C, \label{BoundepS}
\end{gather}
where $C$ is a deterministic constant, independent of $\ep$.}

\begin{theorem} \label{Limiteqequilib} Let $\Psi^{0, \ep}$ be defined by (\ref{resc1}). Then there exists $\Psi^0 \in L^2 (G, \mathcal{D}(\Omega))\cap
L^\infty (\Omega \times G)
$ such that
\begin{gather}
\Psi^{0, \ep} \xrightarrow{s. 2-s.m.} \Psi^0 (x, \omega),
\label{Convequilib1}
   \\
  \ep \nabla \Psi^{0, \ep} \xrightarrow{s. 2-s.m.} \nabla_\omega \Psi^0 (x, \omega),
\label{Convequilib2} 
\end{gather}
The limit function $\Psi^0$ is the unique solution to the variational equation
\begin{gather}
\int_{\mathcal{F}} \nabla_\omega \Psi^0  \cdot \nabla_\omega g \ d\mu + \beta  \int_{\mathcal{F}} n_H( \Psi^0 )  g \ d\mu =
-N_\sigma \int_{\mathcal{F}} \big( \frac{\sigma}{C_0} h(d (\omega)) \nabla_\omega d(\omega) \cdot \nabla_\omega g + \notag \\ \frac{\sigma}{C_0} h(d(\omega)) \Delta_\omega d(\omega) g +
 \frac{\sigma}{C_0} h'(d(\omega)) |\nabla_\omega d|^2 g + \frac{\nabla_\omega \sigma}{C_0} \cdot \nabla_\omega d  h(d) g \big) \ d\mu, \quad \forall g\in \mathcal{D}(\Omega)\cap L^\infty (\Omega).
\label{Eqvareq1}
\end{gather}
\end{theorem}

\begin{proof} Using a priori estimates (\ref{EnergyepS})-(\ref{BoundepS}) and the stochastic two-scale convergence in the mean compactness theorem 3.7 from \cite{BMW}
and Proposition \ref{twosccompact},
 we conclude that, after taking a proper subsequence, the sequences $\{ \Psi^{0, \ep} \}$ and $\{ \ep \nabla  \Psi^{0, \ep} \}$ have stochastic two-scale limits in the mean $\Psi^0$ and $\nabla_\omega \Psi^0$. Furthermore, $\displaystyle \chi_{G_f^\ep (\omega)} $ converges in stochastic two-scale in the mean toward $\chi_{\mathcal{F}}$.

Because of the lower-semicontinuity with respect to the stochastic two-scale convergence in the mean of the $L^q$-norms, $1<q<+\infty$, and estimate (\ref{BoundepS}), the $L^q (G\times \mathcal{F})$-norms of $\Psi^0$ are bounded uniformly with respect to $q$. Hence, $\Psi^0\in L^\infty (\Omega \times G)$, with the same constant as in the bound (\ref{BoundepS}).

Let now $\zeta \in C^\infty_0 (G)$ and $g\in  \mathcal{D}(\Omega)\cap L^\infty (\Omega)$ . Let $g^\ep (x, \omega) = g (\mathcal{T} (\frac{x}{\ep}) \omega) \zeta (x)$, $\sigma^\ep (x, \omega) = \sigma (\mathcal{T} (\frac{x}{\ep})\omega)$  and $d^\ep (x, \omega) =d(\mathcal{T} (\frac{x}{\ep})\omega)$. Using Minty's lemma we write the scaled back
problem (\ref{BP1monaux}) in the equivalent form
\begin{gather}
\int_{\Omega} \int_G \ep \chi_{G_f^\ep (\omega)} \nabla g^\ep \cdot  \ep \nabla ( g^\ep -\Psi^{0,\ep})   \ dx d\mu + \beta  \int_{\Omega} \int_G \chi_{G_f^\ep (\omega)}  n_H(g^\ep )  (g^\ep - \Psi^{0,\ep}) \ dx d\mu +\notag \\
N_\sigma  \int_{\Omega} \int_G \big( \frac{\sigma^\ep }{C_0} h(d^\ep) \ep \nabla d^\ep \cdot \ep \nabla (g^\ep - \Psi^{0,\ep})+  \frac{\sigma^\ep}{C_0} (h(d^\ep) \ep^2\Delta d^\ep   + h'(d^\ep) |\ep \nabla d^\ep|^2) (g^\ep-\Psi^{0,\ep}) +\notag \\
 \frac{\ep \nabla \sigma^\ep}{C_0} \cdot \ep \nabla d^\ep  h(d^\ep) (g^\ep - \Psi^{0,\ep}) \big) \ dx d\mu \geq 0, \quad \forall g\in \mathcal{D}(\Omega)\cap L^\infty (\Omega), \; \zeta \in C^\infty_0 (G).
\label{EqvareqEP1}
\end{gather}
Passing to the limit $\ep \to 0$ is now straightforward (see for instance \cite{BAP98} and \cite{BAP03}). It yields
\begin{gather}
\int_G  \int_{\mathcal{F}}  \nabla_\omega g (\omega) \zeta (x) \cdot  \nabla_\omega (g (\omega) \zeta (x)  -\Psi^{0})   \ dx d\mu + \beta   \int_G \int_{\mathcal{F}}  n_H(g (\omega) \zeta (x) )  ( g (\omega) \zeta (x)- \Psi^{0}) \ dx d\mu +\notag \\
N_\sigma   \int_G \int_{\mathcal{F}} \big( \frac{\sigma (\omega) }{C_0} h(d (\omega))  \nabla_\omega d (\omega) \cdot  \nabla_\omega (g (\omega) \zeta (x) - \Psi^{0})+\notag \\
  \frac{\sigma (\omega)}{C_0} (h(d (\omega)) \Delta_\omega d (\omega)   + h'(d (\omega)) | \nabla_\omega) d (\omega)|^2) (g (\omega) \zeta (x) -\Psi^{0}) +\notag \\
 \frac{ \nabla_\omega \sigma (\omega)}{C_0} \cdot  \nabla_\omega d (\omega)  h(d (\omega)) (g (\omega) \zeta (x) - \Psi^{0}) \big) \ dx d\mu \geq 0, \quad \forall g\in \mathcal{D}(\Omega)\cap L^\infty (\Omega), \; \zeta \in C^\infty_0 (G).
\label{EqvareqEP2}'
\end{gather}
Using again Minty's lemma we obtain that $\Psi^0$ satisfies problem (\ref{Eqvareq1}). Due to the strict convexity, $\Psi^0$ is unique and the whole sequence converges.  Moreover, $\Psi^0$ does not depend on $x$, and, by construction,
the function $\Psi^0\big(\mathcal{T} (x/ \ep) \omega\big)$ satisfies a.s. the {
Poisson-Boltzmann equation and the Neumann condition in  \eqref{BP1}. }
\end{proof}
\begin{remark} For passing to the stochastic 2-scale limits for more complicated problems with convex structure, we refer to Hudson et al \cite{HLL20}.\end{remark}
\section{Homogenization }\label{Passlimit}
{

In Subsection \ref{PerioPB} we solved
the nonlinear Poisson-Boltzmann equation, for the
equilibrium potential $\Psi^{0, \varepsilon} (x)$. It allowed computation of the equilibrium  concentrations $n^{0,\ep}_j(x) = n_j^c \exp \{ -z_j \Psi^{0, \varepsilon} (x) \} .$ Furthermore, we established that, as $\varepsilon \to 0$, $\Psi^{0, \varepsilon} (x)$ converges  stochastically  two-scales to $\Psi^0 (\omega)$, the unique solution of the variational problem (\ref{Eqvareq1}). Since
the goal of this section  is to homogenize
the system of linearized equations (\ref{LIN2})-(\ref{EPPR4a}) and of Section \ref{sec4} to establish  Onsager's relationship between the fluxes and the gradients of potentials in the {\bf bulk}, we  make
 a further simplification of the original system and replace in the linearized system the function $n_j^{0, \varepsilon} $ with
 $n_\varepsilon^j (x)=n_j^c \exp \{ -z_j  \Psi^0\big(\mathcal{T} (x/ \ep) \omega\big) \}$.

}

The formal two-scale asymptotic expansion method follows the periodic case  
(see Looker \& Carnie \cite{LC:06}).
The fast variable is now $y=\mathcal{T} (x/ \ep) \omega$ and the expansion of the solutions of
(\ref{LIN2})-(\ref{EPPR4a}) now reads
$$
   \left\{
     \begin{array}{l}
       \dsp \mathbf{u}^\ep (x) = \mathbf{u}^0 (x,y) + \ep \mathbf{u}^1 (x,y) +\dots ,  \\
       \dsp P^\ep (x) = p^0 (x) + \ep p^1 (x,y) + \dots , \\
       \dsp \Phi^\ep_j (x) = \Phi^0_j (x) + \ep \Phi^1_j (x,y) +\dots .
     \end{array}
   \right.
$$

We do not dwell on formal expansions and start by introducing the functional spaces related to the velocity field and the ionic potentials:
%
$$
\mathcal{H}^\ep = \{ \mathbf{z} \in H^1_0 (G_f^\ep (\omega)  )^n , \; \mbox{ div } \mathbf{z}=0 \; \hbox{ in } \;  G_f^\ep (\omega) \} , \quad W^\ep =\{ z \in H^1 (G_f^\ep (\omega) ) , \; z=0 \mbox{ on } 
\p G
  \}.
$$
Then, summing the variational formulation of  (\ref{EPPR4L})-(\ref{EPPR4a}) with
that of  (\ref{LIN2})-(\ref{LIN4}) (weighted by $z_j^2 / \pej$) yields a.s in $\omega$:
\begin{gather}
\mbox{ Find } \; \mathbf{u}^\ep \in \mathcal{H}^\ep \; \mbox{ and } \; \{ \Phi_j^\ep \}_{j=1, \dots , N}  \in ( W^\ep)^N,
 \notag \\
\mathbf{a}\left( ( \mathbf{u}^\ep , \{ \Phi_j^\ep \} ) , ( \xi, \{ \phi_j \} ) \right) := \ep^2 \int_{G_f^\ep (\omega)} \nabla \mathbf{u}^\ep : \nabla \xi \ dx 
 +\sum_{j=1}^N z_j \int_{G_f^\ep (\omega)}  \big( \mathbf{u}^\ep \cdot \nabla \phi_j - \xi \cdot \nabla \Phi_j ^\ep \big) n^j_\ep \ dx +\notag \\  \sum_{j=1}^N \frac{z_j^2}{\pej} \int_{G_f^\ep (\omega)} n^j_\ep \nabla \Phi_j^\ep \cdot \nabla \phi_j \ dx = < \mathcal{L}, ( \xi, \{ \phi_j \} ) > := \notag \\
\sum_{j=1}^N z_j \int_{G_f^\ep (\omega)} n^j_\ep \mathbf{E}  \cdot
\left( \xi - \frac{z_j}{\pej} \nabla \phi_j \right) \, dx - \int_{G_f^\ep (\omega)} \mathbf{f}^*\cdot\xi \, dx ,
\label{VAREP}
\end{gather}
for any test functions $\xi \in \mathcal{H}^\ep$ and $\{ \phi_j \}_{j=1, \dots , N} \in (W^\ep)^N$.\vskip1pt
{
We recall that the concentrations $n^{0,\ep}_j$ are replaced with the statistically homogeneous concentrations $n_\varepsilon^j$.} \vskip1pt
Before studying problem (\ref{VAREP}), we briefly discuss   Poincar\'e inequality in $G_f^\ep (\omega)$. {
For a general class of random domains, it was studied in Beliaev \& Kozlov \cite{BK96}.


With assumptions {\bf R1.--R5.}, the proof of this inequality is analogous to that in the periodic case (see Allaire \cite[Sec 3.1.3, Lemma 1.6]{All97}):
\begin{lemma}\label{Poinc} Under assumptions {\bf R1.--R5.} a.s. in $\omega$,
\begin{equation}
\label{eq.poincare}
|| \xi ||_{L^2 (G_f^\ep (\omega))^n} \leq C \ep ||\nabla  \xi ||_{L^2 (G_f^\ep (\omega))^{n^2}}, \quad \forall \xi\in 
{
H^1_0 (G_f^\ep (\omega))^n,}
\end{equation}
where $C$ is a deterministic constant.
\end{lemma}
\begin{proof}
First, we rescale $\xi (x)$ to ${\tilde \xi}$ being defined on  {
$\displaystyle \varepsilon^{-1} G$}
Next, we extend  $\xi (x)$ by zero to the complement of $G_f (\omega)$.   Let $F_j (\omega)$ be a subset of $F(\omega)$ contained the points having $M_j (\omega)$ as the closest matrix block. This way we obtain a tessellation of the whole space. Now we have  Poincar\'e inequality for every domain $F_j (\omega)\cup {\overline M_j (\omega)}$, with a deterministic constant independent of $j$.
Hence, we have Poincar\'e's inequality for all $j\in \mathcal{J} (\ep)$. Next we add  the complement of the closure of the union of all domains $F_j (\omega)\cup {\overline M_j (\omega)}$, with $j\in \mathcal{J} (\ep)$, in $\displaystyle \varepsilon^{-1} G$. It yields Poincar\'e's inequality in {
$\displaystyle \varepsilon^{-1} G$ for ${\tilde \xi}$,} with deterministic constant independent of $\ep$. Rescaling back with respect to $\ep$, gives inequality (\ref{eq.poincare}).
\end{proof}
\begin{proposition}\label{APRIORI1}
Let us assume {\bf R1.--R5.} and let $\mathbf{E}{
 =\nabla \Psi^{ext}}  $ and $ \mathbf{f}^*$ be
given elements of $L^2 (G)^n$. Then
variational problem (\ref{VAREP}) admits a unique solution
$\displaystyle ( \mathbf{u}^\ep , \{ \Phi_j^\ep \}_{1\leq j\leq N} ) \in \mathcal{H}^\ep \times (W^\ep )^N$.
Furthermore, there exists a deterministic constant $C$, which does not depend on $\ep$, nor on
$\mathbf{f}^*$ and $\mathbf{E} $, such that
the solution satisfies the following a priori estimates
\begin{gather}
|| \mathbf{u}^\ep ||_{L^2 (G_f^\ep (\omega))^n} +\ep ||\nabla  \mathbf{u}^\ep ||_{L^2 (G_f^\ep (\omega))^{n^2}} \leq C\bigg( || \mathbf{E}  ||_{L^2 (G )^n} +  ||  \mathbf{f}^* ||_{L^2 (G )^n} \bigg)
 \label{AprioriVelocity} \\
\hskip-10pt \max_{1\leq j\leq N} || \Phi_j^\ep ||_{H^1 (G_f^\ep (\omega))} \leq C \bigg( || \mathbf{E}  ||_{L^2 (G)^n} + ||  \mathbf{f}^* ||_{L^2 (G )^n} \bigg).
 \label{AprioriPotential}
\end{gather}
\end{proposition}

\begin{proof}
The Cauchy-Schwartz inequality yields continuity of the bilinear form $a$ and the linear form $\mathcal{L}$  on $\mathcal{H}^\ep \times (H^1 (G_f^\ep (\omega) )/\mathbb{R})^N$. Furthermore for $\xi =\mathbf{u}^\ep$ and $\phi_j =\Phi_j^\ep$, we find out that the second integral (the cross-term) in the definition of $a$ cancels.
Next, because of the $L^\infty$-bounds on $\Phi^{0, \ep}$,  $n^j_\ep\geq C>0$, for a deterministic constant $C$, and the bilinear form $\displaystyle a( ( \mathbf{u}^\ep , \{ \Phi_j^\ep \}_{1\leq j\leq N}) , ( \mathbf{u}^\ep , \{ \Phi_j^\ep \}_{1\leq j\leq N} ) )$ is  $\mathcal{H}^\ep \times (H^1 (G_f^\ep (\omega) )/\mathbb{R})^N$-elliptic.  Now, the Lax-Milgram lemma implies existence and uniqueness of solution of  problem (\ref{VAREP}).\vskip4pt
The a priori estimates (\ref{AprioriVelocity})-(\ref{AprioriPotential}) follow by testing the problem (\ref{VAREP}) by the solution, using the $L^\infty$-estimate for $\Psi^0$ and using Poincar\'e's inequality (\ref{eq.poincare}).
\end{proof}

As in Subsection \ref{PerioPB}, to simplify the presentation we use an extension operator from
the perforated domain $G_f^\ep (\omega)$ into $\O$ (although
it is not necessary). Using hypothesis {\bf R1.-R4.}, in analogy with the periodic case, (studied  for instance in Acerbi et al \cite{ACPDMP92}, Cionarescu \& Saint-Jean-Paulin \cite{CioSJP79} and Jikov et al \cite{JKO94}), there exists  an extension operator
$T^\ep$ from $H^1(G_f^\ep (\omega))$ in $H^1(G)$ satisfying
$T^\ep \phi|_{G_f^\ep (\omega)}=\phi$ and the inequalities
$$
\|T^\ep\phi\|_{L^2(G)}\le C\|\phi\|_{L^2(G_f^\ep (\omega))},\,
\|\nabla(T^\ep\phi)\|_{L^2(G)}\le C\|\nabla\phi\|_{L^2(G_f^\ep (\omega))}
$$
with a deterministic constant $C$ independent of $\ep$, for any $\phi\in H^1(G_f^\ep (\omega))$.
We keep for the extended function $T^\ep \Phi^\ep_j$ the same notation $\Phi^\ep_j$.

We extend $\mathbf{u}^\ep$ by zero in
$G\backslash G_f^\ep (\omega)$. It is well known that extension by
zero preserves $L^q$ and $W^{1,q}_0$ norms for $1<q<\infty$.
Therefore, we can replace $G_f^\ep (\omega)$ by $G$ in estimate
(\ref{AprioriVelocity}).

The pressure field $P^\ep$ is reconstructed using de Rham's theorem, see Temam \cite{TemamNS}.
It is thus unique up to an additive constant.
The {\it a priori} estimates for the pressure are not easy to obtain and in the case of periodic porous media require Tartar's construction from \cite{Ta1980} (see also Allaire \cite{ALL89} or  Allaire \cite[Sec 3.1.3]{All97}). Here we deal with a random porous medium and the pressure extension was constructed only for checkerboard type random domains in Beliaev \& Kozlov \cite{BK96}.
 Nevertheless, assumptions {\bf R1.-R4.} allow to construct a "security domain" $Y_j (\omega)$ of the fixed deterministic size surrounding every $M_j (\omega)$, $j\in \mathcal{J} (\ep)$. It is such that its distance to neighboring solid inclusions $M_{\ell}$ is bigger than a strictly positive deterministic constant. Then we repeat Tartar's construction  of the restriction operator, developed originally for periodic porous media  (see Allaire \cite{All97}), for every $j\in \mathcal{J} (\ep)$. Next, by gluing all the pieces, the restriction operator is defined as a continuous operator $R : H^1_0 (\frac{1}{\ep }G)^n \to H^1_0 (G_f (\omega))^n$. Note that if div $\varphi =0$ in $G/\ep$, then div $(R_\ep \varphi)=0$ in $G_f (\omega)$. Rescaling in exactly the same way as in the periodic case yields the restriction operator $R_\ep : H^1_0 (G)^n \to H^1_0 (G_f^\ep (\omega))^n$, such that div $\varphi =0$ in $G$ implies div $(R_\ep \varphi)=0$ in $G_f^\ep (\omega)$. $\nabla P^\ep$ is then extended using duality, as in the periodic case, and an  extended pressure is ${\tilde P}^\ep$ is obtained and the following estimate holds
\begin{equation}\label{Estpressgrad}
 | \langle \nabla {\tilde P}^\ep , \varphi \rangle_{H^{-1} (G), H^1_0 (G)}|\leq ( ||\varphi ||_{L^2(G)^n} + \ep ||\nabla \varphi ||_{L^2(G)^{n^2}} ), \quad \forall \varphi \in H^1_0 (G).
\end{equation}
Furthermore, a slight modification of the argument from  Avellaneda \& Lipton \cite{LA} gives that
 the pressure extension ${\tilde P}^\ep$ is given by
\begin{equation}\label{1.27}
\tilde P^\ep=
\left\{
\begin{array}{ll}
    \dsp P^\ep & \hbox{ in } G_f^\ep (\omega),  \\
  \dsp  \frac{1}{\vert \ep Y_{i} (\omega)\vert}\,\int_{\ep Y_{i} (\omega) }\,P^\ep &
\hbox{ in }  \ep Y_i (\omega), 
  \end{array}
\right.
\end{equation}
for each $i\in \mathcal{J} (\ep) $. The results are summarized in
\begin{lemma} \label{C1.6}
Let $\tilde P^\ep$ be defined by (\ref{1.27}). Then (a.s.) in $\omega$ it
satisfies the estimates
\begin{gather}
\Vert {\tilde P}^\ep - \frac{1}{
\vert G \vert} \int_{G} {\tilde P}^\ep  dx \Vert_{L^2 (G)}
 \leq C \bigg( || \mathbf{E}  ||_{L^2 (G )^n} + ||  \mathbf{f}^* ||_{L^2 (G )^n} \bigg) , \notag \\
    \Vert \nabla {\tilde P}^\ep\Vert_{ H^{-1} (G)^n} \leq C \bigg( || \mathbf{E}  ||_{L^2 (G )^n} + ||  \mathbf{f}^* ||_{L^2 (G )^n} \bigg) .
\notag
\end{gather}
\end{lemma}

Using the a priori estimates and the notion of two-scale convergence,
we are able to prove our main convergence result.

\begin{theorem}
\label{1.15} Let  us assume {\bf R1.--R5.}
Let {
$n_j^0 = n_j^c \exp \{ -z_j \Psi^0 \}$} and
$\{ \mathbf{u}^\ep , \{ \Phi_j^\ep \}_{j=1, \dots , N} \}$
be the variational solution of (\ref{VAREP}).
We extend the velocity $\mathbf{u}^\ep$ by zero in $G \setminus G_f^\ep (\omega)$ and the pressure $P^\ep$
by ${\tilde P}^\ep$, given by (\ref{1.27}) and normalized by $\int_{G \setminus G_f^\ep (\omega)} \tilde P^\ep =0$.
Then there exist limits $(\mathbf{u}^0 , P^0) \in V \times L^2_0 (G)$
and $\{ \Phi_j^0 , \Phi_j^1  \}_{j=1, \dots , N} \in \left(H^1_0 (G)\times L^2(G ; X)\right)^N$
such that the following convergences hold
\begin{gather}
 \mathbf{u}^\ep \to  \mathbf{u}^0 (x,\omega)  \qquad \hbox{ in the stochastic two-scale sense } \label{1.69} \\
\ep \nabla \mathbf{u}^\ep  \to  \nabla_\omega \mathbf{u}^0
 (x, \omega)  \qquad \hbox{ in the stochastic two-scale sense }
\label{1.70}\\
  {\tilde P}^\ep
 \to  P^0 (x)  \hbox{ strongly in } \, L^2_0 (G), \; \mbox{(a.s.) in} \; \omega, \label{1.71} \\
\Phi_j^\ep  \to  \Phi_j^0(x) \ \hbox{in the stochastic two-scale sense } \,
\label{convPhi}\\
 \chi_{G_f^\ep (\omega)} \nabla \Phi_j^\ep  \to \chi_{\mathcal{F}} (\omega)  \{ \nabla_x \Phi_j^0(x) + \Phi_j^1(x, \omega) \} \,
\hbox{ in the stochastic two-scale sense. }
\label{convgradPhi}
\end{gather}
In addition,   for $  j=1, \dots , N$,
\begin{equation}\label{propert}
     \chi_{\mathcal{M}} (\omega) \Phi_j^1(x, \omega) =0,  \quad \chi_{\mathcal{M}} (\omega) \mathbf{u}^0 (x,\omega) =0 \quad \mbox{and} \quad P^0 (x, \omega) =P^0 (x) \quad \mbox{a.e. on } \; G\times \Omega.
\end{equation}
Furthermore, $(\mathbf{u}^0 , P^0,\{ \Phi_j^0 , \Phi_j^1 \}_{j=1, \dots , N})$ is the unique solution
of the two-scale homogenized problem
\begin{gather}
- \Delta_\omega \mathbf{u}^0(x,\omega) + \nabla_\omega p^1(x,\omega) = - \nabla_x P^0 (x) - \mathbf{f}^*(x) \notag \\
+\sum_{j=1}^N z_j n_j^0 (\omega)(\nabla_x \Phi_j^0(x) +  \Phi_j^1(x,\omega) + \mathbf{E} (x))
\ \mbox{ in } \, G\times \mathcal{F} , \label{Stokes1} \\
\div_\omega \mathbf{u}^0 (x,\omega) =0 \ \mbox{ in } \, G \times \mathcal{F} , \ \mathbf{u}^0 (x, \omega) =0 \ \mbox{on } \, G \times \mathcal{M} , \label{Stokes2}\\
\div_x \left(\mathbb{E} (\mathbf{u}^0 ) \right) =0 \, \mbox{ in } G, \quad \mathbb{E} (\mathbf{u}^0 ) \cdot \nu =0 \, \mbox{and} \; \Phi_j^0=0 \mbox{ on } \p G,\label{VAREP51}\\
- \div_\omega \Big( n^0_j (\omega) \big(  \Phi^1_j(x,\omega)  + \nabla_x \Phi_j ^0 (x) +  \mathbf{E} (x)  + \frac{\pej}{z_j} \mathbf{u}^0 \big) \Big) =0  \; \mbox{ in } G\times \mathcal{F} ,\label{Diff1}\\
{
\quad \mbox{curl}_\omega \ \Phi^1_j =0  \; \mbox{ in } G\times \mathcal{F} ,
\label{Diff111}} \\
  n^0_j (\omega) ( \Phi^1_j  + \nabla_x \Phi_j ^0  + \mathbf{E}) =0 \quad  \mbox{ in } \; G \times \mathcal{M}, \label{Diff2} \\
-\div_x  \mathbb{E} ( n^0_j \big( s \Phi^1_j  + \nabla_x \Phi_j^0 +  \mathbf{E} (x)  +\frac{\pej}{z_j} \mathbf{u}^0 \big)  ) =0  \; \mbox{in } \, G  ,\label{Diffg2}
\end{gather}
for  $j=1, \dots , N$.
\end{theorem}

\begin{remark}Following the terminology of Allaire \cite{All97}, the limit problem introduced in Theorem \ref{1.15} is called the two-scale, two-pressure homogenized problem. It is well posed because the two incompressibility constraints
(\ref{Stokes2}) and (\ref{VAREP51}) are exactly dual to the two pressures $P^0(x)$
and $p^1(x, \omega)$ which are their corresponding Lagrange multipliers.

 The separation of scales from the above two-scale limit problem and
extracting the purely macroscopic homogenized problem will be done later
in Proposition \ref{prop.eff}, Section \ref{sec4}.\end{remark}

The proof of Theorem \ref{1.15}  will follow from several auxiliary lemmas.

We start by rewriting the variational formulation (\ref{VAREP}) with
a velocity test function which is not divergence-free, so we can still
take into account the pressure
\begin{gather}
\ep^2 \int_{G_f^\ep (\omega)} \nabla \mathbf{u}^\ep : \nabla \xi \, dx -
\int_{G_f^\ep (\omega)} P^\ep \, \mbox{div} \, \xi \, dx + 
 \sum_{j=1}^N \int_{G_f^\ep (\omega)} z_j  \big( -\xi \cdot  \nabla \Phi_j ^\ep + \mathbf{u}^\ep \cdot\nabla \phi_j \big)n^j_\ep \ dx +\notag \\  \sum_{j=1}^N \frac{z_j^2}{\pej} \int_{G_f^\ep (\omega)} n^j_\ep \nabla \Phi_j^\ep \cdot \nabla \phi_j \ dx =
- \sum_{j=1}^N \frac{z_j^2}{\pej} \int_{G_f^\ep (\omega)} n^j_\ep  \mathbf{E}  \cdot \nabla \phi_j \ dx \notag \\
+\sum_{j=1}^N \int_{G_f^\ep (\omega)} z_j n_j ^\ep \mathbf{E}  \cdot \xi \, dx
- \int_{G_f^\ep (\omega)} \mathbf{f}^*\cdot\xi \, dx , \label{VAREP1}
\end{gather}
for any test functions $\xi \in H^1_0 (G_f^\ep (\omega))$ and $\phi_j \in W^\ep$,  $1\leq j\leq N$. Of course, one keeps the divergence constraint $\div \mathbf{u}^\ep =0$ in $G_f^\ep (\omega)$.
Next we define the two-scale test functions:
\begin{gather}
\xi^\ep (x) = \xi ( x, \mathcal{T} (\frac{x}{\ep}\omega) ), \ \xi \in C^\infty_{0} (G ; \mathcal{D} (\Omega)^n) ,  \,
 \ \xi =0 \; \mbox{ on } \; G \times \mathcal{M}  , \, \mbox{div}_\omega \xi (x,\omega) =0 \; \mbox{ on } \; G\times \mathcal{F}, \label{Ksi}\\
\phi_j^\ep = \varphi_j (x) + \ep \gamma_j (x, \mathcal{T} (\frac{x}{\ep})\omega ) , \ \varphi_j  \in C^\infty_{0} (G),
  \, \gamma_j \in C^\infty_{0} (G ; \mathcal{D} (\Omega)), \; j=1, \dots N.\label{phij}
\end{gather}

\begin{lemma}\label{pressnonosc} Let us suppose the assumptions of Theorem \ref{1.15} and convergences (\ref{1.69})-(\ref{convgradPhi}). Then any cluster point $\{ \mathbf{u}^0 , P^0 \}$ satisfies (\ref{propert}).
\end{lemma}
\begin{proof} If we take $\xi$ which is with support in $\mathcal{M}$, then passing to the two-scale limit immediately gives $\chi_{\mathcal{M}} (\omega) \mathbf{u}^0 (x,\omega) =0$.
 Next we take as test function $\xi^\ep = \ep \xi ( x, \mathcal{T} (\frac{x}{\ep}\omega) )$, where $\xi$ is given by (\ref{Ksi}) and $\phi_j=0$, for each $j$, then passing to the two-scale limit gives
 $$ 0=\int_G \int_\Omega P^0 \mbox{div}_\omega \ \xi (x, \omega) \ d\mu dx.$$
  Remark \ref{rem_3} and the ergodicity assumption on $\mathcal{F}$ yields
 $P^0 (x, \omega) =P^0 (x) \quad \mbox{a.e. on } \; G\times \Omega$.  For a detailed computation see  Wright \cite[Lemma 2.4]{wright}.
\end{proof}
\begin{lemma}\label{momentumflo} Let us suppose the assumptions of Theorem \ref{1.15} and convergences (\ref{1.69})-(\ref{convgradPhi}). Then any cluster point $\{ \mathbf{u}^0 , P^0 , \{ \Phi_j^0 , \Phi_j^1 \} _{j=1, \dots N}\}$ satisfies incompressibility constraints (\ref{Stokes2})-(\ref{VAREP51}) and the variational equation
\begin{gather}
 \int_{G \times \mathcal{F} }  \nabla_\omega \mathbf{u}^0 (x,\omega) : \nabla_\omega \xi  \ dx d\mu -
\int_{G \times \mathcal{F} } P^0 (x)  \, \div_x \xi  \ dx d\mu  +\notag \\
 \sum_{j=1}^N \int_{G \times \mathcal{F}} z_j n_j^0 (\omega) \Big( -\xi (x,\omega) \cdot
(\nabla_x  \Phi_j^0 (x)+  \Phi_j^1 (x,\omega) )
 +  \mathbf{u}^0 (x,\omega) \cdot (\nabla_x \varphi_j(x) + 
 g_j (x,\omega)) \Big) \ dx d\mu \notag \\
+ \sum_{j=1}^N \frac{z_j^2}{\pej} \int_{G \times \mathcal{F} } n_j^0 (\omega) (\nabla_x \Phi_j^0 (x) +
 \Phi_j^1(x,\omega)) \cdot (\nabla_x \varphi_j  + 
 g_j )  \ dx d\mu
= \notag \\
- \sum_{j=1}^N \frac{z_j^2}{\pej} \int_{G \times \mathcal{F}} n_j^0 (\omega)  \mathbf{E} (x) \cdot (\nabla_x \varphi_j (x) +
 g_j (x,\omega) ) \ dx d\mu
+ \sum_{j=1}^N \int_{G \times \mathcal{F}} z_j n_j ^0 (\omega)  \mathbf{E} (x) \cdot  \xi \ dx d\mu\notag \\
- \int_{G \times \mathcal{F}} \mathbf{f}^*(x)\cdot\xi(x,\omega) \, dx d\mu ,  \label{VAREP3}
\end{gather}
for any test functions $\xi$ given by (\ref{Ksi}) and $ \phi_j $ given by (\ref{phij}). {
Notice that $\nabla_\omega \gamma_j$ was replaced with the element from the corresponding closed subspace: $g_j \in L^2(G; X)$ .}
\end{lemma}
\begin{proof} If we multiply div$\ \mathbf{u}^\ep =0$ by $\xi^\ep$, integrate over $G_f^\ep (\omega)\times \mathcal{F}$ and pass to the two-scale limit, incompressibility constraint (\ref{Stokes2}) follows immediately.

The incompressibility constraint (\ref{VAREP51}) follows analogously, but with a choice of test function $\mathbb{E} (\xi^\ep)$ and $\phi_j^\ep =0$ for each $j$.

 Using convergences (\ref{1.69})-(\ref{convgradPhi}) and by recalling that $n^j_\ep = n_j^0 (\mathcal{T} (x/ \ep) \omega)$ we pass to the two-scale limit in equation (\ref{VAREP1}) without difficulty.
\end{proof}

The next step is to prove the well-posedness of variational equation  (\ref{VAREP3}), which by uniqueness of the limit
 automatically
implies that the entire sequence $\displaystyle ( \mathbf{u}^\ep , P^\ep , \{ \Phi_j^\ep \}_{1\leq j\leq N} )$
converges.

 Let the functional space for the velocity $\mathbf{u}^0$ be given by
$$
 V= \{ \mathbf{z}^0(x,\omega) \in L^2 \left(G ; \mathcal{D} (\Omega)^n \right) \; \mbox{ satisfying }
(\ref{Stokes2})-(\ref{VAREP51}) \} ,
$$
\begin{lemma} \label{pressrec}  Let $\eta \in L^2_0 (G)$. Then there exists $\Theta \in V$ such that
\begin{equation}\label{pressx}
  \mbox{div}_x \mathbb{E} \{ \Theta \} =\eta \; \mbox{ in } \; G, \quad \mathbb{E} \{ \Theta \}\cdot \nu =0 \; \mbox{ on } \; \partial G.
\end{equation}
\end{lemma}
\begin{proof}  Let $\mathcal{W} $ be the Hilbert space given by
$$ \mathcal{W} =\{ \ \mathbf{z} \in \mathcal{D} (\Omega)^n \ |\; \div_\omega \mathbf{z} =0 \; \mbox{ in } \; \mathcal{F} \quad \mbox{ and } \quad \mathbf{z} =0 \; \mbox{ on } \; \mathcal{M} \}.$$
We define the random variables $\mathbf{q}^i \in \mathcal{W}$, $i=1, \dots , n$ by
\begin{equation}\label{injperm}
  \int_{\mathcal{F}} \nabla_\omega \mathbf{q}^i : \nabla_\omega \psi \ d\mu +  \int_{\mathcal{F}}  \mathbf{q}^i  \cdot \psi \ d\mu= \int_{\mathcal{F}} \psi_i \ d\mu, \quad \forall \psi \in \mathcal{W}.
\end{equation}
Then we have
\begin{gather*}
  \mathbb{E} \{ q^j_i \} =\int_{\mathcal{F}} \nabla_\omega \mathbf{q}^i : \nabla_\omega  \mathbf{q}^j\ d\mu +  \int_{\mathcal{F}}  \mathbf{q}^i  \cdot  \mathbf{q}^j \ d\mu= \mathbb{E} \{ q^i_j \}
\end{gather*}
and for all $\lambda \in \mathbb{R}^n$
\begin{gather*}
  \sum_{i,j=1}^n \lambda_i \lambda_j \mathbb{E} \{ q^j_i \} =\int_{\mathcal{F}} | \nabla_\omega (\sum_{i=1}^n \lambda_i \mathbf{q}^i )|^2 \ d\mu +\int_{\mathcal{F}} | \sum_{i=1}^n \lambda_i \mathbf{q}^i |^2 \ d\mu .
\end{gather*}
Hence the matrix $\displaystyle K_q = \bigg[ \mathbb{E} \{ q^j_i \} \bigg]_{i,j=1, \dots , n}$ is symmetric positive definite.

Now we set $\Theta = \displaystyle \sum_{i=1}^n \mathbf{q}^i (\omega) \frac{\partial q}{\partial x_i} (x)$, where $q\in H^1 (G)\mathbb{R}$ solves the problem
\begin{equation}\label{injpermpb}
  \mbox{div}_x  \{ K_q \nabla_x q  \} =\eta \; \mbox{ in } \; G, \quad  K_q \nabla_x q \cdot \nu =0 \; \mbox{ on } \; \partial G.
\end{equation}
As $\mathbb{E} \{ \Theta \} =K_q \nabla_x q$, the Lemma is proved.

We notice the analogy with Allaire  \cite[Section 3.1.2]{All97}.
\end{proof}
\begin{proposition}\label{uniq}
Problem (\ref{VAREP3}) with incompressibility constraints
(\ref{Stokes2}) and (\ref{VAREP51}) has a unique solution
\begin{gather}
 ( \mathbf{u}^0 , P^0, \{ \Phi_j^0 , \Phi_j^1 \}_{j=1, \dots , N} ) \in 
 V  \times L^2_0 (G)\times (H^1_0 (G) \times L^2(G; X))^N.\notag
\end{gather}
\end{proposition}

\begin{proof}  We study variational problem (\ref{VAREP3}) with $\xi \in V$, $\varphi_j \in  H^1 (G)/ \mathbb{R}$ and with $\nabla_\omega \gamma_j,$ $j=1, \dots , N$, replaced by arbitrary element of $L^2(G; X)$. We notice that for $\xi \in V$, $\displaystyle \int_{G \times \mathcal{F} } P^0 (x)  \, \div_x \xi  \ dx d\mu =0$. Hence
 we apply the
Lax- Milgram  lemma to prove the existence and uniqueness of
$( \mathbf{u}^0 ,  \{ \Phi_j^0 , \Phi_j^1 \} )$ in
$V  \times  (H^1_0 (G) \times L^2(G; X))^N$.
The only point which requires to be checked is the coercivity of the bilinear form.
We take $\xi =\mathbf{u}^0$, $\varphi_j = \Phi_j^0$ and  $g_j = \Phi_j^1$ as the test functions in (\ref{VAREP3}). Using the incompressibility constraints (\ref{VAREP51}) and the anti-symmetry
of the third integral in (\ref{VAREP3}),
we obtain the quadratic form
\begin{gather}
    \int_{G \times \mathcal{F} } | \nabla_\omega \mathbf{u}^0 (x,\omega) |^2 \ dx d\mu + 
   \sum_{j=1}^N \frac{z_j^2}{\pej} \int_{G \times \mathcal{F} } n_j^0 (\omega) | \nabla_x \Phi_j^0 (x) +
   \Phi_j^1 (x,\omega) |^2 \ dx d\mu .\label{Uniq1}
\end{gather}
Recalling from Lemma \ref{lem.pb} that $n_j^0 (\omega)\geq C>0$ in $\mathcal{F}$, it is easy
to check that each term in the sum on the second line of (\ref{Uniq1}) is bounded
from below by
$$
C \left( \int_{G}  | \nabla_x \Phi_j^0 (x)|^2 \ dx +
\int_{G \times \mathcal{F}} |  \Phi_j^1 (x,y) |^2 \ dx d\mu \right) ,
$$
which proves the coerciveness of the bilinear form in the required space.\vskip0pt
It remains to prove uniqueness of the pressure $p^0$. It is sufficient to prove that for the homogeneous data, $P^0=0$ in $L^2_0 (G)$.

By the above result and using equation (\ref{VAREP3}), we have
$$ 0= \int_{G \times \mathcal{F} } P^0 (x)  \, \div_x \xi  \ dx d\mu =\int_{G}  P^0 (x) \mathbb{E} \{ \mbox{div}_x \ \xi \} \ dx.$$
Hence, by Lemma \ref{pressrec}, $P^0$ is orthogonal to all elements of $L^2_0 (G)$ and, as such, equal to zero.
\end{proof}

\begin{remark}\label{funcspac}  In analogy with Allaire \cite{ALL92a} (see also Allaire \cite[Section 3.1.2]{All97}), the space $V$ is  orthogonal in $L^2$ $(G ; \mathcal{D} (\Omega)^n )$ to
the space of gradients of the form $\nabla_x q(x) + \nabla_\omega q_1(x,\omega)$ with $q(x)\in
H^1 (G)/\RR$ and  $q_1(x,\omega)\in L^2 \left(G\times \mathcal{F}   \right)$.
\end{remark}

{\bf Proof of Theorem \ref{1.15}:}\vskip0pt
By virtue of the a priori estimates in Lemmas \ref{APRIORI1} and \ref{C1.6},
and using the compactness of Proposition \ref{twosccompact} and Lemma \ref{pressnonosc}, there exist a
subsequence, still denoted by $\ep$, and limits
$( \mathbf{u}^0 , p^0 , \{ \Phi_j^0 , \Phi_j^1 \}_{1\leq j\leq N} ) \in V \times L^2_0 (G) \times (H^1_0 (G)\times L^2(G; X))^N$
such that the convergences in Theorem \ref{1.15} are satisfied.
Using Lemma \ref{momentumflo} we pass to the two-scale limit in (\ref{VAREP1}) we get that
the limit $( \mathbf{u}^0 , p^0 , \{ \Phi_j^0 , \Phi_j^1 \}_{1\leq j\leq N} )$
satisfy the  two-scale variational formulation (\ref{VAREP3}).

According to Proposition \ref{uniq}, the limit system has a unique solution and the whole sequence converges.

It remains to recover the two-scale homogenized system (\ref{Stokes1})-(\ref{Diffg2})
from the variational formulation (\ref{VAREP3}).
In order to get the Stokes equations (\ref{Stokes1}) we choose $\varphi_j=0$ and $\gamma_j=0$
in (\ref{VAREP3}). Using Corollary 2.7 from \cite{wright} 
we deduce the existence of a pressure field $p^1(x,\omega)$
in $L^2(G\times \Omega)$ such that
$$
\hskip-10pt - \Delta_\omega \mathbf{u}^0  + \nabla_\omega p^1 = - \nabla_x p^0 - \mathbf{f}^* +
\sum_{j=1}^N z_j n_j^0(\nabla_x \Phi_j^0 +  \Phi_j^1 + \mathbf{E} ) \; \mbox{ in } \, G\times \mathcal{F}.
$$
The incompressibility constraints (\ref{Stokes2}) and (\ref{VAREP51}) are simple
consequences of passing to the two-scale limit in the equation $\div\mathbf{u}^\ep =0$ in $G_f^\ep (\omega)$.
To obtain the cell convection-diffusion equation (\ref{Diff1}) we now choose $\xi=0$ and
$\varphi_j=0$ in (\ref{VAREP3}) while the macroscopic convection-diffusion equation (\ref{Diffg2})
is obtained by taking $\xi=0$ and $\gamma_j=0$.
This finishes the proof of Theorem \ref{1.15}.

\section{Scale separation and Onsager's relations}\label{sec4}

{
The limit problem obtained in Theorem \ref{1.15} contains the two-scales and a large set if unknowns. Furthermore, it is a system of stochastic PDEs in a random geometry.
 For the practical purposes (which are overall the computational ones), it is important to  extract from (\ref{Stokes1})-(\ref{Diffg2}) the macroscopic
homogenized problem, if possible. It requires to separate the slow ($x$-) and fast ($\omega$-) scale.
This was undertaken in Looker \& Carnie \cite{LC:06} for periodic porous media. In Allaire et al \cite{AllMiPi:10} their analysis was simplified
 and   Onsager properties for
the effective fluxes were established. In addition, the scale separation results allowed  establishing further qualitative properties of the effective coefficients and eliminating the fast scale. In this article, our goal is to generalize results from Allaire et al \cite{AllMiPi:10} to stochastic porous media.

The main idea is identifying in  two-scale homogenized problem
(\ref{Stokes1})-(\ref{Diffg2})  the two different sets of macroscopic
fluxes, namely $(\nabla_x P^0 (x)+\mathbf{f}^*(x))$ and
$\{\nabla_x \Phi_j^0(x) + \mathbf{E} (x)\}_{1\leq j\leq N}$.
Therefore, we introduce two families of random geometry  problems, indexed by $k\in\{1,...,n\}$
for each component of these fluxes. We denote by $\{\mathbf{e}^k\}_{1\leq k\leq n}$
the canonical basis of $\RR^n$.}

The first family of random geometry problem, corresponding to the macroscopic pressure gradient, is
\begin{gather}
\mbox{Find} \; \{ \mathbf{v}^{0,k}, \Theta^{0,k}_j \} \in \mathcal{W}\times X, \ j=1, \dots, N, \quad \mbox{such that} \notag \\
\int_{ \mathcal{F} }  \nabla_\omega \mathbf{v}^{0,k} (\omega) : \nabla_\omega \xi (\omega) \  d\mu -
 \sum_{j=1}^N \int_{ \mathcal{F}} z_j n_j^0 (\omega) \Theta^{0,k}_j (\omega) \cdot (\xi (\omega) -\frac{z_j}{\pej}\zeta_j (\omega) ) \ d\mu  \notag \\
 + \sum_{j=1}^N z_j \int_{ \mathcal{F}} n_j^0 (\omega)\mathbf{v}^{0,k} (\omega) \cdot \zeta_j (\omega)  \ d\mu   = \int_{ \mathcal{F}} \mathbf{e}^k \cdot \xi (\omega)\ d\mu, \quad \forall \xi \in \mathcal{W}, \, \zeta_j \in X, j=1, \dots, N.
 \label{StokesAux0}
\end{gather}
The second family of random geometry problem, corresponding to the macroscopic diffusive flux, is
for each species $i\in\{1,...,N\}$
\begin{gather}
\mbox{Find} \; \{ \mathbf{v}^{i,k}, \Theta^{i,k}_j \} \in \mathcal{W}\times X, \ j=1, \dots, N, \quad \mbox{such that} \notag \\
\int_{ \mathcal{F} }  \nabla_\omega \mathbf{v}^{i,k} (\omega) : \nabla_\omega \xi (\omega) \  d\mu -
 \sum_{j=1}^N \int_{ \mathcal{F}} z_j n_j^0 (\omega) \bigg(\Theta^{i,k}_j (\omega) \cdot (\xi (\omega) -\frac{z_j}{\pej}\zeta_j (\omega)) - \mathbf{v}^{i,k} (\omega) \cdot \zeta_j (\omega)\bigg) \ d\mu  \notag \\
 =  z_i \int_{ \mathcal{F}}  \mathbf{e}^k \cdot (\xi (\omega)-\frac{z_i}{\pei}\zeta_i (\omega) )\ d\mu, \quad \forall \xi \in \mathcal{W}, \, \zeta_j \in X, j=1, \dots, N.
  \label{StokesAuxi}
\end{gather}

\begin{lemma} \label{Exaux} Problems (\ref{StokesAux0}) and (\ref{StokesAuxi}) admit a unique solution. \end{lemma}

Then, we can decompose the solution of (\ref{Stokes1})-(\ref{Diffg2}) as
\begin{gather}
\mathbf{u}^0 (x,\omega) = \sum_{k=1}^n \left( - \mathbf{v}^{0,k}(\omega)
\left( \frac{\p p^0}{\p x_k} +
f^*_k \right)(x) +
\sum_{i=1}^N \mathbf{v}^{i,k}(\omega) \left( E^*_k + \frac{\p \Phi^0_i}{\p x_k} \right)(x) \right) \label{micro1}\\
 \Phi_j^1 (x,\omega) = \sum_{k=1}^n  \left( - \Theta^{0,k}_j(\omega)
\left( \frac{\p p^0}{\p x_k} +
f^*_k \right)(x) +
\sum_{i=1}^N \Theta^{i,k}_j(\omega) \left( E^*_k + \frac{\p \Phi^0_i}{\p x_k} \right)(x) \right) .\label{micro3}
\end{gather}
We average (\ref{micro1})-(\ref{micro3}) in order to get a purely
macroscopic homogenized problem. We introduce the
 perturbation of the total electrochemical potential $\mu^\ep_j$ as
$$
\delta \mu_j^\ep = - z_j (\Phi_j^\ep + \Psi^{ext} )
$$
and the ionic flux of the $j$th species
$$
\mathbf{j}_{j}^\ep = \frac{z_j}{\pej}  n^\ep_j \left( \nabla \Phi_j^\ep + \mathbf{E}
+ \frac{\pej}{z_j} \mathbf{u}^\ep \right) .
$$
The corresponding homogenized quantities are defined as
\begin{gather*}
\mu_j(x) = - z_j (\Phi_j^0 (x) + \Psi^{ext}(x) ) ,\\
\mathbf{j}_j(x) = \frac{z_j}{\pej} \mathbb{E} \bigg\{  n^0_j (\omega) ( \nabla_x \Phi_j^0 (x) + \mathbf{E}  +  \Phi^1_j (x,\omega) + \frac{\pej}{z_j} \mathbf{u}^0(x,\omega) ) \bigg\},\quad \mathbf{u}(x) =  \mathbb{E} \{ \mathbf{u}^{0} \} .
\end{gather*}
From (\ref{micro1})-(\ref{micro3}) we deduce the homogenized or upscaled equations
for the above effective fields.

\begin{proposition}\label{prop.eff}
Introducing the flux
$\mathcal{J}(x) = (\mathbf{u},\{\mathbf{j}_j\}_{1\leq j\leq N})$
and the gradient
$\mathcal{F}(x) =(\nabla_x p^0, \{\nabla_x \mu_j\}_{1\leq j\leq N})$,
the macroscopic equations are
\begin{gather}
\div_x \mathcal{J} =0 \quad \mbox{in } \; G , \label{hom1}\\
\mathcal{J} = - \mathcal{B} \mathcal{F} - \mathcal{B} (\mathbf{f}^* , \{ 0\})\label{hom2}
\end{gather}
with a symmetric positive definite $\mathcal{B}$, defined by
\begin{equation}
\label{Onsager}
   \mathcal{B}  = \left(
                    \begin{array}{cccc}
\dsp  \mathbb{K}    & \dsp \frac{\mathbb{J}_1}{z_1}  & \dots  &  \dsp  \frac{\mathbb{J}_N}{z_N}  \\
\dsp  \mathbb{L}_1 & \dsp \frac{\mathbb{D}_{11}}{z_1} & \cdots &  \dsp       \frac{\mathbb{D}_{1N}}{z_N}\\
\dsp  \vdots  & \vdots & \ddots & \vdots \\
\dsp  \mathbb{L}_N & \dsp  \frac{\mathbb{D}_{N1}}{z_1} & \cdots & \dsp \frac{\mathbb{D}_{NN}}{z_N} \\
                    \end{array}
                  \right) ,
\end{equation}
and complemented with the boundary conditions for $p^0$
and $\{\Phi_j^0\}_{1\leq j\leq N}$.
The matrices $\mathbb{J}_i$, $\mathbb{K}$, $\mathbb{D}_{ji}$ and $\mathbb{L}_j$
are defined by their entries
\begin{gather}
\{ \mathbb{J}_i \}_{lk} = \mathbb{E} \{ \mathbf{v}^{i,k} \cdot \mathbf{e}^l \} ,\quad
\{ \mathbb{K} \}_{lk} = \mathbb{E} \{ \mathbf{v}^{0,k} \cdot \mathbf{e}^l \},\notag\\
\{ \mathbb{D}_{ji} \}_{lk} = \mathbb{E} \{ n_j^0  ( \mathbf{v}^{i,k}
+ \frac{z_j}{\pej} \left( \delta_{ij} \mathbf{e}^k +  \Theta^{i,k}_j  \right) ) \cdot \mathbf{e}^l \} ,\quad
\{ \mathbb{L}_j \}_{lk} = \mathbb{E} \{ n_j^0  ( \mathbf{v}^{0,k}
+ \frac{z_j}{\pej}  \Theta^{0,k}_j  ) \cdot \mathbf{e}^l \} .\notag
\end{gather}
\end{proposition}

\begin{remark}
The tensor $\mathbb{K}$ is called permeability tensor, $\mathbb{D}_{ji}$
are the electrodiffusion tensors.
The symmetry of the tensor $\mathcal{M}$ is equivalent to the famous Onsager's
reciprocal relations. It was already proved in \cite{LC:06}. However, the
positive definiteness of $\mathcal{M}$ was proved in \cite{AllMiPi:10}.
It is essential in order to state that (\ref{hom1})-(\ref{hom2}) is an
elliptic system which admits a unique solution.
\end{remark}

\begin{proof}[Proof of Proposition \ref{prop.eff}]
The relation in \eqref{hom1}  is an immediate consequence of \eqref{VAREP51} and \eqref{Diffg2}.
Taking the expectation on the left- and right-hand sides of equalities \eqref{micro1} and  \eqref{micro3} and considering
the definitions of the homogenized functions $\{\mu_j\}$ and $\{\mathbf{j}_j\}$ and of the entries of the matrix $\mathcal{B}$, after elementary computations we arrive at \eqref{hom2}.

In order to justify positive definiteness of $\mathcal{B}$ we fix an arbitrary vector $\eta=\big(\eta^0,\eta^1,\ldots,\eta^N\big)$, $\eta^j\in\mathbb R^d$,  and denote by $\mathbf{v}^\eta$ and $\Theta^\eta_j$ the following functions:
\begin{equation*}\label{lin_comb0}
  \mathbf{v}^\eta=\sum\limits_{l=1}^d\Big\{\eta_l^0\mathbf{v}^{0,l}(\omega)+ \sum\limits_{i=1}^N\eta_l^i
  \mathbf{v}^{i,l}(\omega)\Big\},
\end{equation*}
\begin{equation*}\label{lin_comb1}
  \Theta_j^\eta=\sum\limits_{l=1}^d\Big\{\eta_l^0\Theta^{0,l}_j(\omega)+ \sum\limits_{i=1}^N\eta_l^i
  \Theta^{i,l}_j(\omega)\Big\}.
\end{equation*}
From \eqref{StokesAux0} we derive that
$$
 \{ \mathbf{v}^{\eta}, \Theta^{\eta}_j \} \in \mathcal{W}\times X, \ j=1, \dots, N,
$$
and
\begin{equation}
\label{StokesAux_eta}
\begin{array}{c}
\displaystyle
\int_{ \mathcal{F} }  \nabla_\omega \mathbf{v}^{\eta} (\omega) : \nabla_\omega \xi (\omega) \  d\mu -
 \sum_{j=1}^N \int_{ \mathcal{F}} z_j n_j^0 (\omega) \Theta^{\eta}_j (\omega) \cdot (\xi (\omega) -\frac{z_j}{\pej}\zeta_j (\omega) ) \ d\mu  \\[2mm]
 \displaystyle
 + \sum_{j=1}^N z_j \int_{ \mathcal{F}} n_j^0 (\omega)\mathbf{v}^{\eta} (\omega) \cdot \zeta_j (\omega)  \ d\mu
 +\sum_{j=1}^N\int_{ \mathcal{F} }z_j n_j^0 (\omega)\big(\eta^j\cdot\xi(\omega)-\eta^j\cdot\frac1{\bf Pe}_j\zeta_(\omega)\big)\,d\mu
  \\[5mm]
 \displaystyle = \int_{ \mathcal{F}} \eta^0 \cdot \xi (\omega)\ d\mu, \quad \forall \xi \in \mathcal{W}, \, \zeta_j \in X, j=1, \dots, N.
\end{array}
\end{equation}
Substituting $\mathbf{v}^{\eta}$ for $\xi$ and $ \Theta^{\eta}_j$ for $\zeta_j$ in this integral relation yields
\begin{equation*}
\label{quadr_calb}
\begin{array}{c}
\displaystyle
\int_{ \mathcal{F} }  |\nabla_\omega \mathbf{v}^{\eta} (\omega) |^2 \  d\mu +
 \sum_{j=1}^N \int_{ \mathcal{F}} \frac{z_j^2}{{\rm Pe}_j} n_j^0 (\omega) |\Theta^{\eta}_j (\omega)|^2  \, d\mu   \\[2mm]
 \displaystyle
= \int_{ \mathcal{F} }\eta^0\cdot\mathbf{v}^{\eta} (\omega)\,d\mu
 + \sum_{j=1}^N \int_{ \mathcal{F}} n_j^0 (\omega)\eta^j\cdot \Big(z_j\mathbf{v}^{\eta} (\omega)-
  \frac{z_i^2}{{\rm Pe}_j}\Theta^{\eta}_j (\omega)\Big)  \, d\mu;
\end{array}
\end{equation*}
here the quadratic form on the left-hand side have been obtained in the same way as the quadratic form in  \eqref{Uniq1}.
This implies the following relation:
 \begin{equation*}
\label{quadr_calbmod}
\begin{array}{c}
\displaystyle
\int_{ \mathcal{F} }  |\nabla_\omega \mathbf{v}^{\eta} (\omega) |^2 \,  d\mu +
 \sum_{j=1}^N \int_{ \mathcal{F}} \frac{z_j^2}{{\rm Pe}_j} n_j^0 (\omega) |\Theta^{\eta}_j (\omega)+\eta^j|^2  \, d\mu   \\[2mm]
 \displaystyle
= \int_{ \mathcal{F} }\eta^0\cdot\mathbf{v}^{\eta} (\omega)\,d\mu
 + \sum_{j=1}^N \int_{ \mathcal{F}} n_j^0 (\omega)\eta^j\cdot \Big(z_j\mathbf{v}^{\eta} (\omega)+
  \frac{z_i^2}{{\rm Pe}_j}\big(\Theta^{\eta}_j (\omega)+\eta^j\big)\Big)  \, d\mu \\[2mm]
 \displaystyle
 \mathbb K\eta^0\cdot\eta^0+\sum\limits_{j=1}^N\mathbb J_j\eta^j\cdot\eta^0+
 \sum\limits_{i,\,j=1}^N\mathbb J_j\eta^j\cdot\eta^0z_i\eta^i\cdot\mathbb D_{ij}\eta^j+
 \sum\limits_{j=1}^N z_j\eta^j\cdot\mathbb L_j\eta^0 \\[2mm]
 \displaystyle
 =\mathcal{B}\big(\eta^0,\, z_j\eta^j\big)^T\cdot \big(\eta^0,\, z_j\eta^j\big)^T,
\end{array}
\end{equation*}
which in turn yields the desired positive definiteness.

In order to show that  $\mathcal{B}$ is symmetric we consider $\mathbf{v}^{\check{\eta}}$ and $\Theta_j^{\check{\eta}}$
with another set of vectors $\check{\eta}^0$, $\{\check{\eta}^j\}_{j=1}^N$ in $\mathbb R^d$.   Then we substitute $\mathbf{v}^{\check{\eta}}$ and $0$ for  $\xi$ and $\zeta_j$, respectively, in  \eqref{StokesAux_eta}. 
In a similar integral relation corresponding to $\check{\eta}^0$, $\{\check{\eta}^j\}_{j=1}^N$ we substitute  $0$ for $\xi$ and
$\Theta^\lambda_j$ for $\zeta_j$. Summing up the resulting relations we get
$$
\begin{array}{c}
\displaystyle
\int_{ \mathcal{F} }\Big( \nabla_\omega \mathbf{v}^{\eta}\cdot\nabla_\omega \mathbf{v}^{\check{\eta}}
+\sum\limits_{j=1}^N n_j^0 (\omega) \Theta^{\eta}_j (\omega)\cdot \Theta^{\check{\eta}}_j (\omega)\Big) \,  d\mu \\[4mm]
\displaystyle
=\int_{ \mathcal{F} } \eta^0\cdot\nabla_\omega \mathbf{v}^{\check{\eta}} \,  d\mu +\sum\limits_{j=1}^N
\int_{ \mathcal{F} }z_j n_j^0 (\omega)\big(\eta^j\cdot \mathbf{v}^{\check{\eta}}-\frac {z_j}{{\rm Pe}_j}{\check{\eta}}
\cdot \Theta^{\eta}_j (\omega) \big) \,  d\mu.
\end{array}
$$
Exchanging $\eta$ and $\check{\eta}$ and considering the symmetry of the integral on the left-hand side we obtain
$$
\begin{array}{c}
\displaystyle
\int_{ \mathcal{F} } \eta^0\cdot\nabla_\omega \mathbf{v}^{\check{\eta}} \,  d\mu +\sum\limits_{j=1}^N
\int_{ \mathcal{F} }z_j n_j^0 (\omega)\big(\eta^j\cdot \mathbf{v}^{\check{\eta}}+\frac {z_j}{{\rm Pe}_j}{{\eta}}
\cdot \Theta^{\check{\eta}}_j (\omega) \big) \,  d\mu\\[4mm]
\displaystyle
=\int_{ \mathcal{F} } \check{\eta}^0\cdot\nabla_\omega \mathbf{v}^{{\eta}} \,  d\mu +\sum\limits_{j=1}^N
\int_{ \mathcal{F} }z_j n_j^0 (\omega)\big(\check{\eta}^j\cdot \mathbf{v}^{{\eta}}+\frac {z_j}{{\rm Pe}_j}{\check{\eta}}
\cdot \Theta^{{\eta}}_j (\omega) \big) \,  d\mu.
\end{array}
$$
This implies the following equality
$$
\eta^0\cdot \mathbb K\check{\eta}^0+\sum\limits_{j=1}^N\eta^0\cdot \mathbb J_j\check{\eta}^j+
\sum\limits_{i=1}^Nz_i\eta^i\cdot\Big(\mathbb L_i\check{\eta}^0+\sum\limits_{j=1}^N\mathbb D_{ij}\check{\eta}^j\Big)
$$
$$
=\check{\eta}^0\cdot \mathbb K{\eta}^0+\sum\limits_{j=1}^N\check{\eta}^0\cdot \mathbb J_j{\eta}^j+
\sum\limits_{i=1}^Nz_i\check{\eta}^i\cdot\Big(\mathbb L_i{\eta}^0+\sum\limits_{j=1}^N\mathbb D_{ij}{\eta}^j\Big).
$$
Therefore,
$$
\mathcal{B}\big(\eta^0, z_1\eta^1, \ldots,z_N\eta^N\big)^t\cdot
\big(\check{\eta}^0, z_1\check{\eta}^1, \ldots,z_N\check{\eta}^N\big)^t=
\mathcal{B}\big(\check{\eta}^0, z_1\check{\eta}^1, \ldots,z_N\check{\eta}^N\big)^t
\cdot\big(\eta^0, z_1\eta^1, \ldots,z_N\eta^N\big)^t
$$
This yields the desired symmetry of the matrix $\mathcal{B}$.
 \end{proof}
 {
\begin{corollary}\label{regularity}
The homogenized equations in Proposition \ref{prop.eff} form
a symmetric elliptic system
\begin{gather}
\div_x \{ \mathbb{K} (\nabla_x p^0+ \mathbf{f}^* ) + \sum_{i=1}^N \mathbb{J}_i (\nabla_x \Phi_i^0 + \mathbf{E} ) \} = 0\; \mbox{ in } \; \O , \notag \\
\div_x \{ \mathbb{L}_j (\nabla_x p^0+ \mathbf{f}^* ) + \sum_{i=1}^N \mathbb{D}_{ji} (\nabla_x \Phi_i^0 + \mathbf{E} ) \} = 0\; \mbox{ in } \; \O , \notag
\end{gather}
with  boundary conditions (\ref{VAREP51}).
In particular, $\mathbf{E}$ and $ \mathbf{f}^* \in H^1 (G)^n$, with $\sigma \in C^1 (\partial F(\omega))$ bounded, imply that the pressure field $P^0 \in H^2 (G)$.
\end{corollary}}
\section{\label{StrongS} Strong convergence and correctors}
{
Besides the  stochastic two-scale convergence of the microscopic fluxes and pressures to the effective ones,
we also prove  convergences of the energies.

First, we recall that the $L^2$-norm squared is lower semi continuous with respect to  the  stochastic two-scale convergence. In our situation, where after Corollary \ref{regularity} the limit functions are smooth, it can be seen through a simple direct argument. First, we notice  that
the formulas of scale separations (\ref{micro1}) and (\ref{micro3}) imply that the functions $\mathbf{u}^0 (x, \mathcal{T} (\frac{x}{\ep} ) \omega )$ and $\Phi_j^1 (x, \mathcal{T} (\frac{x}{\ep} )\omega )$
are admissible. Hence from
\begin{gather*}  \ep^2 \int_{\Omega} \int_G  | \nabla \mathbf{u}^\ep |^2 \ dx d\mu \geq \ep^2 \int_{\Omega} \int_G  | \nabla \mathbf{u}^0 (x, \mathcal{T} (\frac{x}{\ep} ) \omega ) |^2 \ dx d\mu +\\
 2\int_{\Omega} \int_G   \ep \nabla \mathbf{u}^0 (x, \mathcal{T} (\frac{x}{\ep} ) \omega ) \cdot \ep \nabla \big( \mathbf{u}^\ep - \mathbf{u}^0 (x, \mathcal{T} (\frac{x}{\ep} ) \omega )\big) \ dx d\mu \end{gather*}
and  passing to the limit $\ep \to 0$ gives
$$  \lim_{\ep \to 0}  \ep^2 \int_{\Omega} \int_G  | \nabla \mathbf{u}^\ep |^2 \ dx d\mu \, \geq \int_{\Omega\times \mathcal{F}} | \nabla_\omega \mathbf{u}^0 (x, \omega) |^2  \ d\mu dx. $$
Similarly,
$$  \lim_{\ep \to 0}  \int_{\Omega}
 \int_{G^\ep_f (\omega)}  n^j_\ep | \nabla \Phi_j^\ep |^2 \ dx d\mu
\geq \int_{\Omega \times \mathcal{F}} n_j^0 (\omega) | \nabla_x \Phi_j^0 (x) +
   \Phi_j^1 (x, \omega) |^2  \ dx d\mu .$$
A stronger result is
\begin{proposition}\label{convenerg}
We have the for $j=1,\dots , N$,
\begin{gather}
 \lim_{\ep \to 0}  \ep^2 \int_{\Omega} \int_G  | \nabla \mathbf{u}^\ep |^2 \ dx d\mu \, =\int_{\Omega\times \mathcal{F}} | \nabla_\omega \mathbf{u}^0 (x, \omega) |^2  \ d\mu dx ,\label{energu}\\
 \lim_{\ep \to 0}  \int_{\Omega}
 \int_{G^\ep_f (\omega)}  n^j_\ep | \nabla \Phi_j^\ep |^2 \ dx d\mu
= \int_{\Omega \times \mathcal{F}} n_j^0 (\omega) | \nabla_x \Phi_j^0 (x) +
   \Phi_j^1 (x, \omega) |^2  \ dx d\mu . \label{energd}
\end{gather}
\end{proposition}
\begin{proof}
We follow the  proof from the periodic case (Allaire  \cite[Theorem 2.6]{All92} and Allaire et al \cite[Sec 5]{AllMiPi:10}).
We start from the energy equality corresponding to the variational equation (\ref{VAREP}):
\begin{gather}
\ep^2 \int_{G} |\nabla \mathbf{u}^\ep |^2  \ dx +\sum_{j=1}^N \frac{z_j^2}{\pej} \int_{G^\ep_f (\omega)} n_\ep^j | \nabla \Phi_j^\ep |^2 \ dx =
- \sum_{j=1}^N \frac{z_j^2}{\pej} \int_{G^\ep_f (\omega)} n_\ep^j  \mathbf{E} \cdot \nabla \Phi_j^\ep  \, dx +\notag \\
\sum_{j=1}^N z_j \int_{G^\ep_f (\omega)} n_j ^\ep \mathbf{E} \cdot \mathbf{u}^\ep \, dx
- \int_{G^\ep_f (\omega)} \mathbf{f}^*\cdot\mathbf{u}^\ep \, dx .
\label{Energepsilon}
\end{gather}
For the homogenized variational problem (\ref{VAREP3}) the energy equality reads
\begin{gather}
  \int_{\Omega\times \mathcal{F}} | \nabla_\omega \mathbf{u}^0 |^2  \ dx d\mu  + \sum_{j=1}^N \frac{z_j^2}{\pej} \int_{\Omega\times \mathcal{F}} n_j^0 (\omega) | \nabla_x \Phi_j^0 (x) +  \Phi_j^1 (x, \omega) |^2  \ dx d\mu =
- \sum_{j=1}^N \frac{z_j^2}{\pej} \int_{\Omega\times \mathcal{F}} n_j^0 (\omega ) \mathbf{E} \cdot \notag \\
\cdot (\nabla_x \Phi_j^0 (x)  +  \Phi_j^1 (x, \omega) ) \, dx d\mu
+ \sum_{j=1}^N z_j \int_{\Omega\times \mathcal{F}} n_j ^0 (\omega)\mathbf{E} \cdot \mathbf{u}^0 (x, \omega) \, dx d\mu
 - \int_{\Omega\times \mathcal{F}} \mathbf{f}^* \cdot \mathbf{u}^0 (x, \omega) \, dx d\omega .\label{Energehomog}
\end{gather}
In (\ref{Energepsilon}) we observe the convergence of the right-hand side to the right-hand side of  (\ref{Energehomog}). Next we use the lower semicontinuity, with respect to the  stochastic two-scale convergence,  of the left-hand side
and the equality (\ref{Energehomog}) to conclude (\ref{energu})-(\ref{energd}).
\end{proof}

\begin{theorem}\label{thmstrong}
Under the assumptions of Section \ref{Passlimit}, the following strong two-scale convergences hold
\begin{equation}\label{STVvelocity}
\lim_{\ep \to 0}\int_{\Omega} \int_G \left| \mathbf{u}^\ep(x) - \mathbf{u}^0 (x, \mathcal{T} (\frac{x}{\ep} )\omega)) \right|^2 \, dx d\mu =0
\end{equation}
 and
\begin{equation}\label{STspecies}
\lim_{\ep \to 0} \int_{\Omega} \int_{G^\ep_f (\omega)} \left| \nabla \left(   \Phi_j^\ep (x) - \Phi_j^0 (x)
   \right) - \Phi_j^1 (x, \mathcal{T} (\frac{x}{\ep} ) \omega) \right|^2 \,dx  d\mu=0.
\end{equation}
\end{theorem}

\begin{proof}
We have
\begin{gather}
\int_{\Omega} \int_G  \ep^2 |\nabla [\mathbf{u}^0 (x, \mathcal{T} (\frac{x}{\ep} ) \omega)]
- \nabla \mathbf{u}^\ep (x) |^2 \, dx d\mu
= \int_{\Omega} \int_G  |[\nabla_y \mathbf{u}^0] (x, \mathcal{T} (\frac{x}{\ep} ) \omega) |^2 \, dx d\mu \notag \\
+ \int_{\Omega} \int_G \ep^2 |\nabla \mathbf{u}^\ep (x) |^2 \, dx d\mu
-  2 \int_{\Omega} \int_G \ep [\nabla_y \mathbf{u}^0] (x, \mathcal{T} (\frac{x}{\ep} ) \omega) \cdot \nabla \mathbf{u}^\ep (x) \, dx d\mu
+ O(\ep) . \label{Veloccorr}
\end{gather}
Using Proposition \ref{convenerg} for the second term in the right-hand side of (\ref{Veloccorr})
and passing to the two-scale limit in the third term in the right-hand side of (\ref{Veloccorr}),
we deduce
$$
\lim_{\ep \to 0}\int_{\Omega} \int_G \ep^2 \left| \nabla  \left(\mathbf{u}^\ep (x) - \mathbf{u}^0 ( x, \mathcal{T} (\frac{x}{\ep} ) \omega)\right)\right|^2 \ dx d\mu=0
$$
Now application of  Poincar\'e inequality (\ref{eq.poincare}) in $G_f^\ep (\omega)$ yields (\ref{STVvelocity}).

On the other hand, by virtue of Theorem \ref{PoissBoleq}, $n^\ep_j$ is uniformly
positive, i.e., there exists a constant $C>0$, which does not depend on $\ep$, such that
\begin{gather}
\int_{\Omega} \int_{G^\ep_f (\omega)}\int \left| \nabla \left(   \Phi_j^\ep (x) - \Phi_j^0 (x)
   \right) - \Phi_j^1 (x, \mathcal{T} (\frac{x}{\ep} ) \omega )\right|^2 \,dx d\mu \leq C \notag\\
\int_{\Omega} \int_{G^\ep_f (\omega)} n^\ep_j \left| \nabla \left(   \Phi_j^\ep (x) - \Phi_j^0 (x)
   \right)- \Phi_j^1 (x, \mathcal{T} (\frac{x}{\ep} ) \omega )\right|^2 \,dx d\mu.\label{dilatpot}
\end{gather}
Developing the right-hand side of (\ref{dilatpot}) as we just did for the velocity
and using the fact that $n_\ep^j(x)=n^0_j( \mathcal{T} (\frac{x}{\ep} ) \omega)$ is a two-scale test function,
we easily deduce (\ref{STspecies}).
\end{proof}
}


\begin {thebibliography} {99}

\bibitem{ACPDMP92}  E. Acerbi, V. Chiad\`o Piat, G. Dal Maso, D. Percivale,  An extension theorem from connected sets, and homogenization in general periodic domains. Nonlinear Analysis: Theory, Methods and Applications, 18(5) (1992), 481--496.
\bibitem{AM:03}
P. Adler, V. Mityushev,
\newblock Effective medium approximation and exact formulae
for electrokinetic phenomena in porous media,
J. Phys. A: Math. Gen. {\bf 36} (2003), 391-404.

\bibitem{ALL89}  Allaire, G.:  Homogenization of the Stokes Flow
in a Connected
    Porous Medium,  Asymptotic Analysis, {\bf 2} (1989), 203-222.
\bibitem{All97}  Allaire, G.: One-Phase Newtonian Flow, in \it
Homogenization and Porous Media \rm , ed. U.Hornung, Springer,
New-York, (1997), 45-68.
\bibitem{All92} Allaire, G.:
Homogenization and two-scale convergence,  SIAM J. Math. Anal.,
{\bf 23} (1992),  1482-1518.
\bibitem{ALL92a} Allaire, G.:
Homogenization of the unsteady Stokes equations in porous media,
in "Progress in partial differential equations: calculus of variations,
applications" Pitman Research Notes in Mathematics Series {\bf 267}, 109-123,
C. Bandle et al. eds, Longman Higher Education, New York (1992)



\bibitem{AllMiPi:10} G. Allaire,  A. Mikeli\'c,  A. Piatnitski,
\newblock Homogenization of the linearized ionic transport
equations in rigid periodic porous media,
\newblock {\em J. Math. Phys.} {\bf 51}, \rm 123103 (2010); doi:10.1063/1.3521555\rm 2010.

\bibitem{beta}  G. Allaire, J. F. Dufr\^eche,  A. Mikeli\'c,  A. Piatnitski,  Asymptotic analysis of the Poisson-Boltzmann equation describing electrokinetics in  porous media,  Nonlinearity, Vol. 26 (2013), p. 881--910.

\bibitem{AllBrizzDufMiPi:12}  G. Allaire, R. Brizzi, J. F. Dufr\^eche,  A. Mikeli\'c,  A. Piatnitski,
    Ion transport in  porous media: derivation of the macroscopic equations using up-scaling and properties
    of the effective coefficients,  Comput. Geosci, Vol. 17 (2013), no. 3, p.  479-496.

\bibitem{AllBrizzDufMiPi:14}  G. Allaire, R. Brizzi, J. F. Dufr\^eche,  A. Mikeli\'c,  A. Piatnitski,
 Role of  non-ideality for the  ion transport in  porous media: derivation of the macroscopic equations using upscaling,  Phys. D, 282 (2014), p. 39-60.

\bibitem{ABDM2015}  G. Allaire, O. Bernard, J.-F. Dufr\^eche, A. Mikeli\'c, Ion transport through deformable porous media: derivation of the macroscopic equations using upscaling,  Comp. Appl. Math., Vol. 36 (2017), 1431-- 1462.




\bibitem{BK96} A.Y. Beliaev, S.M.  Kozlov,  Darcy equation for random porous media. Communications on pure and applied mathematics, 49(1) (1996), 1--34.



\bibitem{BMW} A.Bourgeat, A.Mikeli\'c, S.Wright, On the Stochastic
       Two-Scale Convergence in the Mean and Applications,
       \it Journal f\"ur die reine und angewandte Mathematik ( Crelles
       Journal ), \rm Vol. 456 (1994), pp. 19 -- 51.
\bibitem{BAP98} A.Bourgeat, A.Mikeli\'c, A.Piatniski,
       Mod\`ele de double porosit\'e al\'eatoire \rm ,    \it C. R. Acad. Sci.
      Paris, S\'erie I, Math. , \rm t. 327 (1998),  p. 99-104.
\bibitem{BAP03} A. Bourgeat, A. Mikeli\'c , A. Piatnitski, On the double porosity model of single phase flow in random media,
 \it Asymptotic Analysis , \rm Vol. 34 (2003), p. 311-332.

\bibitem{Cardenas} H.E. Cardenas, L.J. Struble,  Electrokinetic nanoparticle treatment of hardened cement paste for reduction of
permeability. J. Mater. Civil Eng.18(4) (2006), 554--560.

\bibitem{CioSJP79} D. Cioranescu, J. Saint-Jean-Paulin, Homogenization in open
sets with holes, J. Math. Anal. Appl. Vol. 71 (1979), 590--607.

\bibitem{CSTA:96}
D. Coelho, M. Shapiro, J.F. Thovert, P. Adler,
Electro-osmotic phenomena in porous media,
J. Colloid Interface Sci. {\bf 181} (1996), 169-90.

\bibitem{Dorm:03}
L. Dormieux, E. Lemarchand, O. Coussy,
Macroscopic and Micromechanical Approaches to
the Modelling of the Osmotic Swelling in Clays, Transport in Porous Media, {\bf 50} (2003), 75-91.

\bibitem{Dufreche05}
{J.-F. Dufr\^{e}che, O. Bernard, S. Durand-Vidal, P. Turq}.
\newblock Analytical theories of transport in concentrated electrolyte
 solutions from the msa.
\newblock {\em J. Phys.Chem. B}, 109:9873, (2005).


\bibitem{ern} Ern, A., Joubaud, R.,  Leli\`evre, T. (2012). Mathematical study of non-ideal electrostatic correlations in equilibrium electrolytes. Nonlinearity, 25(6), 1635.

\bibitem{GT} D. Gilbarg, N. S. Trudinger, Elliptic partial differential equations of second order. Springer, 1983.


\bibitem{GCA:06}
A.K. Gupta, D. Coelho and P. Adler,
Electroosmosis in porous solids for high zeta potentials,
Journal of Colloid and Interface Science {\bf 303} (2006), 593-603.



\bibitem{HLL20} Hudson, T., Legoll, F.,  Leli\`evre, T. (2020). Stochastic homogenization of a scalar viscoelastic model exhibiting stress-strain hysteresis. ESAIM: Mathematical Modelling and Numerical Analysis, 54(3), 879--928.


\bibitem{JKO94} V.V. Jikov,  S.M. Kozlov,  O.A. Oleinik, {\em  Homogenization of differential operators and integral functionals}, Springer Science and Business Media, 1994.

\bibitem{KBA:05}
G. Karniadakis, A. Beskok, N. Aluru,
Microflows and Nanoflows. Fundamentals and Simulation.
Interdisciplinary Applied Mathematics, Vol. 29, Springer, New York, (2005).



\bibitem{LA} Lipton, R. and  Avellaneda, M.:  A Darcy Law for Slow
Viscous Flow Past a Stationary Array of Bubbles,  Proc. Royal Soc. Edinburgh {\bf 114A},
   (1990), 71-79.


\bibitem{LC:06}
J.R. Looker, S.L. Carnie,
Homogenization of the ionic transport
equations in periodic porous media. Transp. Porous
Media {\bf 65} (2006), 107-131.

\bibitem{Lyklema} J. Lyklema,
Fundamentals of Interface ans Colloid Science, Academic Press, Vol 2 (1995).

\bibitem{Mahmoud} A. Mahmoud, J. Olivier, J. Vaxelaire, A. F. A.  Hoadley,  Electrical field: a historical review of its application and contributions in wastewater sludge dewatering. Water Res.44(8) (2010), 2381--2407.



\bibitem{MSA:01}
S. Marino, M. Shapiro, P. Adler,
Coupled transports in heterogeneous media,
J. Colloid Interface Sci. {\bf 243} (2001), 391-419.



\bibitem{MM:02}
C. Moyne, M. Murad,
Electro-chemo-mechanical couplings in swelling clays
derived from a micro/macro-homogenization procedure, Int. J. Solids Structures {\bf 39} (2002),
6159-6190.

\bibitem{MM:03}
C. Moyne, M. Murad,
Macroscopic behavior of swelling porous media derived
from micromechanical analysis, Transport Porous Media {\bf 50} (2003), 127-151.

\bibitem{MM:06}
C. Moyne, M. Murad,
A Two-scale model for coupled electro-chemomechanical
phenomena and Onsager's reciprocity relations in expansive clays: I Homogenization
analysis, Transport Porous Media {\bf 62} (2006), 333-380.

\bibitem{MM:06a}
C. Moyne, M. Murad,
A two-scale model for coupled
electro-chemo-mechanical phenomena and Onsager's reciprocity
relations in expansive clays: II. Computational validation.
Transp. Porous Media {\bf 63(1)} (2006), 13-56.

\bibitem{MM:08}
C. Moyne, M. Murad,
A dual-porosity model for ionic solute
transport in expansive clays, Comput Geosci {\bf 12} (2008), 47-82.


\bibitem{NGU} Nguetseng, G.:
A general convergence result for a functional related to the
theory of homogenization, SIAM J. Math. Anal. {\bf20}(3), 608--623 (1989).

\bibitem{OBW:78}
R. W. O'Brien, L. R. White,
Electrophoretic mobility of a spherical colloidal particle,
J. Chem. Soc., Faraday Trans. {\bf 2 74(2)} (1978), 1607-1626.

\bibitem{Ottosen} L.M. Ottosen, I.V. Christensen, I. Rorig-Dalgard, P.E. Jensen, H.K., Hansen,  Utilization of electromigration
in civil and environmental engineering - Processes, transport rates and matrix changes. J. Environ. Sci.
Health A-Toxic/Hazard. Subst. Environ. Eng. 43(8) (2008), 795--809.





\bibitem{Rayetal:12} N. Ray, A. Muntean, P. Knabner,
Rigorous homogenization of a Stokes--Nernst--Planck--Poisson system,
Journal of Mathematical Analysis and Applications, Vol. {\bf 390} (2012), 374-393.

\bibitem{Rayetal:12b} N. Ray, T. van Noorden, F. Frank, P. Knabner,
{\em Multiscale Modeling of Colloid and Fluid Dynamics in Porous Media
Including an Evolving Microstructure,}
Transp. Porous Med. {\bf 95} (2012), pp. 669-696.





\bibitem{schmuck09} M. Schmuck,  Analysis of the Navier--Stokes--Nernst--Planck--Poisson system. Mathematical Models and Methods in Applied Sciences, 19(06) (2009), 993--1014.

\bibitem{schmuck} M. Schmuck,
{\em Modeling And Deriving Porous Media Stokes--Poisson--Nernst--Planck
Equations By A Multiple-Scale Approach,}
Comm. Math. Sci. {\bf 9} (2011), 685-710.

\bibitem{schmuck2} M. Schmuck,
{\em First error bounds for the porous media approximation of
the Poisson-Nernst-Planck equations,}
ZAMM Z. Angew. Math. Mech. {\bf 92} (2012), pp. 304-319.

\bibitem{schmuckBaz} M. Schmuck, M. Z. Bazant,  Homogenization of the Poisson--Nernst--Planck equations for ion transport in charged porous media. SIAM Journal on Applied Mathematics, 75(3) (2015), 1369--1401.

\bibitem{Ta1980}   Tartar, L.:
Convergence of the Homogenization Process, \it
   Appendix of \rm 
  E.~Sanchez-Palencia.
\newblock {\em Nonhomogeneous media and vibration theory},
Vol. {\bf 127} {\em Lecture Notes in Physics},
\newblock Springer-Verlag, Berlin, (1980).
\bibitem{TemamNS} Temam, R.:  Navier Stokes equations, North Holland, Amsterdam
(1977).
\bibitem{wright} S. Wright, On the Steady-State Flow of an Incompressible Fluid through a Randomly Perforated Porous Medium, Journal of Differential Equations, Vol. 146 (1998), p. 261-286.
%
\bibitem{ZhPi} V. V. Zhikov, A. L. Pyatnitskii, Homogenization of random singular
structures and random measures, {\sl Izvestiya: Mathematics}, {\bf 70}(3)  (2006),  p.19--67.





\end{thebibliography}
\end{document}